\documentclass[letterpaper,10pt]{amsart}
\usepackage{indentfirst} % Indent first paragraph after section header
\usepackage{amssymb}
\usepackage{amsmath}
\usepackage{mathabx} % makes many symbols (as $\leq$) more beautiful
\usepackage{amsthm}
\usepackage{thmtools}
\usepackage{bbm}             %to use $\mathbbm{1}$
\usepackage[colorlinks=true]{hyperref}  % Links in the pdf
\usepackage[usenames,dvipsnames]{xcolor} % With XCOLOR, printed quality is good, (but with COLOR it's bad)
\usepackage{tikz}
\usepackage{comment}
\usetikzlibrary{arrows,intersections}

\declaretheorem{theorem}
\declaretheorem{lemma}
\declaretheorem{corollary}
\declaretheorem{proposition}

\declaretheorem{definition}

\declaretheorem{remark}
\declaretheorem[name=Acknowledgements,numbered=no]{ack}

\newcommand{\M}{\mathcal{M}}
\newcommand{\R}{\mathbb{R}}
\newcommand{\Z}{\mathbb{Z}}
\newcommand{\N}{\mathbb{N}}

\renewcommand{\P}{\mathbb{P}}

\newcommand{\cU}{\mathcal{U}}

\def\phi{\varphi}
\def\R{{\mathbb R}}

\def\N{{\mathbb N}}
\def\Z{{\mathbb Z}}

\def\H{{\mathbb H}}

\def\P{{\mathcal P}}

\def\F{{\mathcal F}}

\def\M{{\mathcal M}}

\def\le{\leqslant}
\def\ge{\geqslant}

\def\F{\mathcal{F}}
\def\M{\mathcal{M}}

\begin{document}

\title{Phase transitions for geodesic flows and the geometric potential }
%\date{\today}

\author[A. Velozo]{Anibal Velozo}  \address{Princeton University, Princeton NJ 08544-1000, USA.}
\email{avelozo@math.princeton.edu}

\begin{abstract}
In this paper we study the phenomenon of phase transitions for the geodesic flow on some geometrically finite negatively curved manifolds. We define a class of potentials \emph{going slowly to zero through the cusps of} $M$ for which  the pressure map exhibits a phase transition. By a careful choice of the metric at the cusp we construct a geometrically finite manifold for which the geometric potential (or unstable Jacobian) exhibits a phase transition. Our results apply, in particular,  to the geodesic flow on an $M$-puncture sphere, for every $M\ge 3$, and a suitable choice of Riemannian metric.
\end{abstract}

\maketitle

\section{Introduction}
The study of phase transitions is a central topic in statistical mechanics (see \cite{ru2}, \cite{h}). As a result of  external conditions certain properties of our medium might change dramatically, and often discontinuously.
%discontinuously: we are in the presence of a phase transition.
A classical example of phase transition is the passage  between the solid, liquid and gaseous phases of matter, due to a continuous increase of temperature. From a mathematical standpoint phase transitions are strongly related to the regularity, or more precisely, to the lack of regularity of the thermodynamic free energy as a function of  other thermodynamical variables. This is usually called the Ehrenfest  classification of phase transitions (for a historical discussion see \cite{jae}). Thermodynamic formalism, as a subject of dynamical systems, is a mathematical theory inspired by statistical mechanics. This theory has had a considerable impact on the study of the ergodic theory of hyperbolic differentiable dynamical systems (see \cite{si}, \cite{ru}, \cite{bob}). It is in this framework that we will investigate the existence of phase transitions for the geodesic flow on some negatively curved manifolds. 

Before explaining our results--and their motivation--let us introduce some notation. Let $(M,g)$ be a complete pinched negatively curved manifold, and denote by $(g_t)_{t\in \R}$ the geodesic flow on the unit tangent bundle $T^1M$. In this paper we will constantly refer to real valued functions with domain $T^1M$ as `potentials'; this is a standard convention in thermodynamic formalism. The topological pressure of a potential $F$ is the quantity $$P(F)=\sup_{\mu\in \M(g)} \{h_\mu(g)+\int Fd\mu\},$$
where the supremum runs over the space of invariant probability measures of the geodesic flow, which we denote by $\M(g)$ (for precise definitions see Section \ref{pre}). An invariant probability measure satisfying $P(F)=h_\mu(g)+\int Fd\mu$, is called an equilibrium state for the potential $F$. We refer to the map $t\mapsto P(tF)$ as the pressure map of the potential $F$. In this context the parameter $t$ is interpreted as the inverse of the temperature and the thermodynamic free energy is the quantity $P(tF)$ (for a complete discussion about this interpretation we refer the reader to \cite{ru}). Since for the geodesic flow on non-compact manifolds the regularity of the pressure map is not well understood (our best result on the regularity of the pressure is Theorem \ref{deriva}), we will use a more qualitative notion of phase transition (see Definition \ref{phdef}). 

We emphasize that the potentials for which we will describe phase transitions are H\"older continuous. We remark that in the compact case  H\"older potentials can not develop phase transitions.   By the work of Bowen \cite{b73} and Ratner \cite{rat}, we know that the geodesic flow on a compact negatively curved manifold can be modelled as a suspension flow over a shift of finite type.  As a consequence we obtain that the pressure map of a H\"older potential is real analytic, and that every H\"older potential has a unique equilibrium state (for a complete discussion we refer the reader to \cite{b73}, \cite{bob} and \cite{ru2}). In the context of symbolic dynamics the study of phase transitions has a long story. In this situation a phase transition is understood as a time $t$ where the pressure map has discontinuous first, or some higher derivative. For shifts of finite type and non-H\"older potentials phase transitions were constructed in \cite{ho} and \cite{lo}. For countable Markov shifts phase transitions have been extensively  studied by Sarig (see \cite{s1}, \cite{s2} and \cite{s3}). Based on the work of Sarig, phase transitions have also been studied for some interval maps (see for instance  \cite{pz}, \cite{bi}, \cite{it}, \cite{it2}). For suspension flows over countable Markov shifts phase transitions have been studied in \cite{ij}. It worth pointing out that   for suspension flows over countable Markov shifts it is possible to construct phase transitions even when the shift map satisfies the BIP property (see \cite[Theorem 4.1]{ij}). Sarig  proved that on a countable Markov shift satisfying the BIP property, the pressure map of a locally H\"older potential  is real analytic whenever finite (see \cite[Corollary 4]{s4}). This is an significant difference between the thermodynamic formalism of the suspension flow and the shift map. The phase transitions constructed in this paper exhibit a similar behaviour to those from \cite{ij}. Phase transitions for quadratic-like maps have been extensively studied in \cite{cr1}, \cite{cr2}, \cite{cr3} and \cite{cr4}. In this body of work Coronel and Rivera-Letelier also constructed quadratic-like maps having sensitive dependence of the geometric potential at zero and positive temperature. Sensitive dependence is an interesting phenomenon that has not been studied for the geodesic flow.

In this paper we will construct a family of geometrically finite manifolds (see Definition \ref{geofindef}) for which it is possible to construct potentials that exhibit phase transitions. %In Ehrenfest classification a first order phase transition occurs when the first derivative of the pressure is not continuous. 
The philosophy behind our construction of phase transitions is very simple: if a potential decays very slowly to zero through the cusps of our manifold, then it is likely to develop a phase transition. A similar principle was used in the construction of phase transitions for Pomeu-Manneville maps (see \cite{pm}, \cite{lo}, \cite{s2}). The class of potentials vanishing at infinity is denoted by $C_0(T^1M)$; those are potentials that tend to zero at infinity (see Definition \ref{C_0}). In this paper we will define the family of strongly positive recurrent potentials (SPR potentials for short, see Definition \ref{sprdef}). SPR potentials in $C_0(T^1M)$ admit an equilibrium state (see Theorem \ref{eeq2}), and we can compute the first derivative of the pressure map (assuming that our potential is H\"older continuous)  for a large class of directions (see Theorem \ref{deriva}). One of the main goals of this paper is to prove the following result. 

\begin{theorem}[=Theorem \ref{teo6}]\label{i1}  There exists a geometrically finite manifold $M$ for which it is possible to construct a non-negative H\"older potentials $F$ such the the following holds. There exists a constant $t_F\in [-1,0)$ for which 
\begin{enumerate}
%\item The quantity $$t_{F}=\sup\{t: P(tF)=\delta_\P\},$$ is equal to $-1$. 
\item The potential $tF$ is strongly positive recurrent for $t>t_F$. 
\item The potential $tF$ does not have an equilibrium state for $t<t_F$. The pressure map is constant in the interval $(-\infty,t_F)$.
\end{enumerate}
In other words, the pressure map of $F$ exhibits a phase transition at $t =t_F$. Moreover, the pressure map is differentiable in $(-\infty, t_{F})\cup (t_{F},\infty)$. With respect to the behaviour at $t=t_{F}$ we have two possibilities:
\begin{enumerate}
\item[(4)] If the potential $t_{F}F$ does not have an equilibrium state, then the pressure map is differentiable everywhere.
\item[(5)]  If $t_{F}F$ has an equilibrium state, then the pressure map is not differentiable at $t=t_{F}$. 
\end{enumerate}

\end{theorem}

The potential constructed in Theorem \ref{i1}  goes \emph{slowly to zero} through the cusp of $M$ (see Definition \ref{class}).  We remark that in Theorem \ref{i1} the  manifold $M$ can be chosen to be hyperbolic of any dimension greater or equal to two. This result should be compared with the phase transitions described in \cite[Theorem 5]{s2} for the renewal shift, and in \cite{ij} for suspension flows; we have a very similar description of the pressure map. We would like to emphasize that phase transitions for the geodesic flow have been previously constructed in \cite{ij} and  \cite{irv} by symbolic methods (for the modular surface and extended Schottky manifolds resp.). Our method is more geometric, and we do not use any symbolic model. An advantage of our approach is that  we can incorporate the geometric potential (after some normalization) to the family of potentials exhibiting phase transitions. The following result will be proven in Section \ref{cons}.

\begin{theorem}[=Theorem \ref{fin}]\label{i2} There exists an extended Schottky manifold for which the geodesic flow has a measure of maximal entropy and the  geometric potential exhibits a phase transition.  
\end{theorem}

The family of extended Schottky manifolds is defined in Section \ref{defesg}. Theorem \ref{i2} is achieved by a careful modification of the metric at the cusp of a hyperbolic manifold. The main difficulty of this construction is that we can not modify the geometric potential without modifying the entire Poincar\'e series that determines the pressure. Our construction is motivated by the work of Peign\'e in \cite{p}.

\begin{remark} In Theorem \ref{i1}  and Theorem \ref{i2} the manifold $M$ can be assumed to be homeomorphic to a $k$-punctured sphere, for any $k\ge 3$. In both cases exactly one puncture is a cusp. In Theorem \ref{i1} we can moreover assume that the metric on $M$ is hyperbolic. In Theorem \ref{i2} the metric is hyperbolic away from a neighborhood of the cusp. 
\end{remark}

\begin{ack}
The author would like to thanks to his advisor G. Tian for his constant support and encouragements. We would also like to thank F. Riquelme for  useful comments at an early stage of this work,  F. Lin for useful comments on how to improve the aesthetics of the  paper, Y. Pesin for telling us about his joint work \cite{pz}, and G. Iommi for telling us about the idea behind phase transitions constructed for Pomeu-Manneville maps. 
\end{ack}

\section{Preliminaries}\label{pre} In this paper $(M,g)$ will always be a complete Riemannian manifold. We define the unit tangent bundle of $M$ as $T^1M=\{v\in TM: ||v||_g=1\}$. Since $(M,g)$ is complete, the geodesic flow on $T^1M$ is well defined for all times; we denote it by $(g_t)_{t\in\R}$. From now on, whenever we say `measure', we really mean `non-negative Borel measure on $T^1M$'. 
The Riemannian metric $g$ makes $M$ into a metric space, the induced distance function (shortest path distance) is denoted by $d$.  Let $\pi :T^1M\to M$ be the canonical projection. We define a metric on $T^1M$--which we still denote by $d$--in the following way \begin{align}\label{distance} d(x,y)=\max_{t\in [0,1]} d(\pi g_t(x),\pi  g_t(y)),\end{align}
for every $x,y\in T^1M$. We emphasize that this is the metric used in all the statement about the geodesic flow (but other possible candidates of metrics are usually uniformly equivalent to $d$, at least in the pinched negatively curved case). 

\begin{definition} The space of invariant probability measures of the geodesic flow is denoted by $\M(g)$.  Given $\mu\in \M(g)$, we denote by $h_\mu(g)$ the measure theoretic entropy of the measure $\mu$ under the map $g_1$ (the time one map of the geodesic flow).
\end{definition}

The mass of a measure $\mu$ is the number $\mu(T^1M)$, and it is denoted by $|\mu|$. We say that a sequence of measures $(\mu_n)_n$ \emph{converges  vaguely} to $\mu$ if for every compactly supported continuous potential $G$ we have that $$\lim_{n\to\infty}\int G d\mu_n=\int G d\mu.$$
We emphasize that a sequence of invariant probability measures  can converges vaguely to a sub-probability measure, in other words, a sequence of probability measures can lose mass. An important fact about the vague topology is that every sequence of invariant probability measures has a subsequence that converges in the vague topology (this follows from the locally compactness of $T^1M$). In this paper we will use the notation $||F||_0$ to denote the supremum norm of a potential $F$, more precisely $$||F||_0=\sup_{x\in T^1M}|F(x)|.$$

\subsection{ Structure of negatively curved manifolds}\label{strucneg}
From now on we will assume that $(M,g)$ is a complete Riemannian manifold of negative sectional curvature. We will moreover assume that $K_g\in [-a,-b]$, for some $a,b>0$, where $K_g$ is the sectional curvature of $M$. We refer to a  manifold satisfying those properties as a \emph{pinched negatively curved manifold}. 

The universal cover of $M$ is denoted by $\widetilde{M}$. The space $\widetilde{M}$ has a natural compactification, the so-called Gromov compactification. As with any compactification, we will add a `boundary' to $\widetilde{M}$, this is the visual boundary $\partial_\infty\widetilde{M}$. The space $\widetilde{M}\cup\partial_\infty\widetilde{M}$ is called the Gromov compactification of $\widetilde{M}$ and it is homeomorphic to a closed ball. We remark that using the ball model of hyperbolic space we can identify this compactification with the round unit ball; this is the picture we will have in mind even if the curvature is not constant.

We denote by $Iso(\widetilde{M})$ the group of isometries on $\widetilde{M}$.  It is well known that every isometry of $\widetilde{M}$ extends to a homeomorphism of $\partial_\infty \widetilde{M}$ and that the fundamental group of $M$ acts (via Deck transformations) isometrically, freely and discontinuously on $\widetilde{M}$. We identify $\pi_1(M)$ with a subgroup $\Gamma<$ $Iso(\widetilde{M})$. In this context there is a complete classification of the elements of $Iso(\widetilde{M})$: an isometry can be elliptic, hyperbolic or parabolic. An \emph{elliptic isometry} fixes a point in $\widetilde{M}$; this type of isometry will not be presented in $\Gamma$ since $M$ is a manifold. A \emph{hyperbolic isometry} fixes (pointwise) exactly two points   at infinity.  The geodesic connecting the fixed points is called the \emph{axis} of the hyperbolic isometry. The axis of a hyperbolic isometry in $\Gamma$ will descend to a closed geodesic on $M$. A \emph{parabolic isometry} fixes exactly one point at infinity. Parabolic subgroups will be important when studying geometrically finite manifolds, they are defined as follows:

\begin{definition}A \emph{parabolic subgroup} of Iso$(\widetilde{M})$ is a group of isometries where each element is parabolic, and they all fix the same point at infinity. A \emph{maximal parabolic subgroup} of $\Gamma$ is a parabolic subgroup that is not strictly contained in another parabolic subgroup of $\Gamma$. 
\end{definition}

\begin{definition} The \emph{limit set} of $\Gamma$, which we denote by $L(\Gamma)$, is the set of accumulation points in $\partial_\infty\widetilde{M}$ of the orbit of a point $x\in \widetilde{M}$ under $\Gamma$. The \emph{conical limit set}, denoted by $L_c(\Gamma)$, is the set of points $\xi\in \partial_\infty\widetilde{M}$ such that there exists a sequence of translates of $x\in \widetilde{M}$ under $\Gamma$ which converges to $\xi$ while staying at bounded distance from a geodesic ray ending at $\xi$. 
\end{definition}

In this paper $\pi_1(M)=\Gamma$ will always be \emph{non-elementary}, in other words it is not generated by one hyperbolic element, nor a parabolic subgroup (see \cite{bow} for several characterizations of this property). In this case we can relate the non-wandering set of the geodesic flow  with the limit set of $\Gamma$. Let $\Omega\subset T^1M$ be the \emph{non-wandering set} of the geodesic flow. We start by recalling the Hopf's parametrization of $T^1\widetilde{M}$. The unit tangent bundle of $\widetilde{M}$ is identified with $$T^1\widetilde{M}=(\partial_{\infty}\widetilde{M}\times\partial_{\infty}\widetilde{M})\setminus \text{(Diagonal)} \times \R,$$ via Hopf's coordinates by sending each vector $\tilde{v}\in T^1\widetilde{M}$ to $(v_-,v_+,b_{v_+}(o,\pi(\tilde{v})))$. Here $v_-$ and $v_+$ are respectively the negative and positive ends at infinity of the oriented geodesic line determined by $\tilde{v}$ in $\widetilde{M}$, and $b_{v_+}(o,\pi(\tilde{v}))$ is the Busemann function defined for $x,y\in\widetilde{M}$ and  $\xi\in\partial_\infty\widetilde{M}$ as
$$b_\xi(x,y)=\lim_{t\to\infty} d(x,\xi_t)-d(y,\xi_t),$$
where $t\mapsto \xi_t$ is any geodesic ray ending at $\xi$. Under this identification, the geodesic flow acts by translation in the third coordinate. The non-wandering set of $T^1M$ corresponds to the projection of $(L(\Gamma)\times L(\Gamma)\setminus (Diagonal))\times \R$ to $T^1M$ (see \cite{ebe}).  Fix a reference point $x_0\in \widetilde{M}$. A set of the form 
\begin{align}\label{horo}B_\xi(r)=\{x\in \widetilde{M}: b_\xi(x_0,x)\ge r\},\end{align}
is called a (closed) \emph{horoball center at} $\xi$. We say that $\xi$ is the base point of the horoball $B_\xi(r)$. We now proceed to define the important class of geometrically finite manifolds. 
\begin{definition} A point $p\in \partial_\infty \widetilde{M}$ is a \emph{bounded parabolic fixed point} of $\Gamma$ if the maximal parabolic subgroup of $\Gamma$ fixing $p$  acts cocompactly in $L(\Gamma)\setminus \{p\}$.
\end{definition}

\begin{definition}[Geometrically finite manifolds]\label{geofindef} We say that $\Gamma<$ Iso$(\widetilde{M})$ is geometrically finite if every point in $L(\Gamma)$ is a conical point or bounded parabolic. We say that $M$ is geometrically finite if $\pi_1(M)$ is geometrically finite.
\end{definition}

\begin{remark}[Bowditch description of geometrically finite manifolds]\label{gff} It follows from \cite{bow} that the non-compact part of the non-wandering set of the geodesic flow (on a geometrically finite manifold) go through the cusps of $T^1M$ and that each cusp is standard, i.e. the quotient of a horoball by the action of the maximal parabolic subgroup of $\pi_1(M)$ fixing the base point of the horoball. Moreover, $\pi_1(M)$ has a finite number of non-conjugate maximal parabolic subgroups (finite number of cusps). 
\end{remark}

\subsection{Thermodynamic formalism}
We start with the definition of the topological pressure of a potential.
\begin{definition}[Topological pressure] Let $F:T^1M\to \R$ be a continuous potential. We define the \emph{topological pressure} of $F$ as 
$$P(F)=\sup_{\mu\in \M(g)}\{h_\mu(g)+\int Fd\mu\}.$$
\end{definition}
 An invariant probability measure satisfying $P(F)=h_\mu(g)+\int Fd\mu$, is called an \emph{equilibrium state} for the potential $F$. 
For compact dynamical systems there is a well known relation--the variational principle--between the topological pressure of a continuous potential and a weighted version of the topological entropy (see \cite{wa}). 
%If $F=0$ one would expect to recover the topological entropy of the geodesic (see Definition \ref{defent}). This is in fact the case (even for non-compact manifolds) as the following result states.
%\begin{theorem}[Variational principle] \cite{op} Let $M$ be a pinched negatively curved manifold. Then $$h_{top}(g)=\sup_{\mu\in \M(g)} h_\mu(g).$$
%\end{theorem}

It turns out that the topological pressure of  a H\"older potential has a nice characterization in terms of some critical exponent. Moreover, with H\"older regularity there exists at most one equilibrium state. Before making precise those results we start with some notation. Given two points $x,y\in\widetilde{M}$, we  denote by $[x,y]$ the oriented geodesic segment  starting at $x$ and ending at $y$. For a function $G:T^1\widetilde{M}\to \R$, we  use the notation $\int_x^y G$ to represent the integral of $G$ over the  tangent vectors to the path $[x,y]$ (in direction from $x$ to $y$). Given a potential $F:T^1M\to \R$, we denote by $\widetilde{F}:T^1\widetilde{M}\to \R$ the function $\widetilde{F}=F\circ p$, where $p$ is the canonical projection $p:T^1\widetilde{M}\to T^1M$.  The following definition was introduced in \cite{pps} (a similar definition was used earlier in \cite{cou}).

\begin{definition}\label{cexdef} Let $F$ be a continuous potential and $\widetilde{F}$ its lift to $T^1\widetilde{M}$. Define the \emph{Poincar\'e series associated to $(\Gamma,F)$} based at $z\in \widetilde{M}$ as
$$P(s,F)=\sum_{\gamma\in\Gamma} \exp\bigg(\int_z^{\gamma z}(\widetilde{F}-s)\bigg).$$
 The \emph{critical exponent of $(\Gamma,F)$} is $$\delta^F_\Gamma=\inf\{\text{s } | \text{ }P(s,F)\text{ is finite}\}.$$We say that the pair $(\Gamma,F)$ is of \emph{convergence type} if $P(\delta_\Gamma^F,F)<\infty$, in other words, the Poincar\'e series converges at its critical exponent. Otherwise we say $(\Gamma, F)$ is of \emph{divergence type}. 
\end{definition}

We use the notation $\delta_\Gamma$ to denote the critical exponent of $(\Gamma,0)$, where $0$ is the zero potential. 
 If $F$ is a H\"older potential, then the critical exponent does not depend on the base point $z$. We also remark that if $F$ is bounded, then $\delta_\Gamma^F$ is finite.

A general procedure due to S. Patterson \cite{pat} and  D. Sullivan \cite{sul} associates to $\Gamma$ a family of conformal measures on $\partial_\infty \widetilde{M}$ of exponent $\delta_\Gamma$, the so-called \emph{Patterson-Sullivan conformal measures} of $\Gamma$. Similar to the construction of the Patterson-Sullivan conformal measures, we can associate to the pair $(\Gamma,F)$ a family of Patterson densities. We briefly recall this procedure (which generalizes the Patterson-Sullivan construction to include H\"older potentials).   A \emph{Patterson density} of dimension $\delta$ for $(\Gamma,F)$ is a family of finite Borel measures $(\sigma_x)_{x\in\widetilde{M}}$ on $\partial_\infty \widetilde{M}$, such that, for every $\gamma\in\Gamma$, for all $x,y\in\widetilde{M}$ and for every $\xi\in\partial_\infty\widetilde{M}$ we have
$$\gamma_\ast\sigma_x=\sigma_{\gamma x} \quad \mbox{and} \quad \frac{d\sigma_x}{d\sigma_y}(\xi)=e^{-C_{F-\delta,\xi}(x,y)},$$
where $C_{F,\xi}(x,y)$ is the \emph{Gibbs cocycle} defined as
$$C_{F,\xi}(x,y)=\lim_{t\to\infty}\int_y^{\xi_t}\widetilde{F}-\int_x^{\xi_t}\widetilde{F},$$
for any geodesic ray $t\mapsto\xi_t$ ending at $\xi$. The limit in the definition of the Gibbs cocycle always exists because the manifold has negative curvature and the potential is H\"older continuous. If $\delta^F_\Gamma<\infty$, then there exists at least one Patterson density of dimension $\delta^F_\Gamma$ for $(\Gamma,F)$, which support lies in the limit set $L(\Gamma)$ of $\Gamma$ (see \cite[Proposition 3.9]{pps}). If $(\Gamma,F)$ is of divergence type then there is an unique Patterson-Sullivan density of dimension $\delta^F_\Gamma$ (see \cite[Corollary 5.12]{pps}). For now on assume that $\delta_\Gamma^F<\infty$, and let $(\sigma_x)_{x\in\widetilde{M}}$ be a Patterson density of dimension $\delta_\Gamma^F$.  Denote by $(\sigma^{\iota}_x)_{x\in \widetilde{M}}$ the Patterson density of dimension $\delta^F_\Gamma$ for $(\Gamma,F\circ \iota)$, where $\iota$ is the flip isometry map $v\mapsto -v$ on $T^1\widetilde{M}$. Using the Hopf parametrisation $v\mapsto (v_-,v_+,t)$ with respect to a base point $o\in\widetilde{M}$, the measure
$$d\tilde{m}(v)=e^{C_{F\circ \iota-\delta^F_\Gamma,v_-}(x_0,\pi(v))+C_{F-\delta^F_\Gamma,v_+}(x_0,\pi(v))}d\sigma^\iota_o(v_-)d\sigma_o(v_+)dt,$$
is independent of $o\in\widetilde{M}$, $\Gamma$-invariant and invariant by the geodesic flow. This induces a measure $m$ on $T^1M$ called the \emph{Gibbs measure associated to the Patterson density} $(\sigma_x)_{x\in \widetilde{M}}$. If $(\Gamma,F)$ is of divergence type this construction is unique and we refer to the resulting measure as the \emph{Gibbs measure associated to }$F$. A fundamental property of $m_F$ is that, whenever finite (and after a normalization to make it a probability measure) it is the unique equilibrium state for the potential $F$. For $F=0$ this construction gives us the so-called \emph{Bowen-Margulis measure}, which we denote by $m_{BM}$. Finally we state one of the main results in \cite{pps}, which is a crucial input in this work. 
\begin{theorem}\label{pps} \cite[Theorem 2.3]{pps} Let $F$ be a bounded H\"older potential. Then 
$$P(F)=\delta_\Gamma^F.$$
Moreover, if there exists a finite Gibbs measure $m_F$ for $(\Gamma,F)$, then  $m_F/||m_F||$ is the unique equilibrium state of $F$. Otherwise there is not equilibrium state.
\end{theorem} 
We remark that this result was obtained by Otal and Peigne in the case $F=0$ (see \cite{op}). The proof of Theorem \ref{pps} follows very closely the proof of \cite[Theorem 1]{op}. If $F=0$ this is simply saying that if the Bowen-Margulis measure is finite, then its normalization $m_{BM}/||m_{BM}||$ is  the measure of maximal entropy of the geodesic flow. Otherwise there is not measure of maximal entropy. It worth mentioning that in the compact case this measure was constructed in two different ways by Bowen \cite{bo72} and Margulis \cite{marg}. Later on Bowen \cite{bo74} proved the uniqueness of the measure of maximal entropy (in the compact case), so both constructions coincide.

We will briefly discuss under what assumptions a H\"older potential admits an equilibrium state, in other words, when we can ensure that the Gibbs measure $m_F$ is finite. From now on we will assume that $M$ is geometrically finite. In this case we have a finite number of non-conjugate maximal parabolic subgroups (see Remark \ref{gff}). Each maximal parabolic $\P$ defines a critical exponent $\delta_\P^F$ (see Definition \ref{cexdef}). The following result was proved in \cite{dop} for the case $F=0$, and later extended by Coudene in \cite{cou} for his definition of critical exponent. For a proof of the version stated here we refer the reader to \cite[Theorem 2.12]{rv}.

\begin{theorem} \label{gapdiv} Let $M$ be a geometrically finite manifold and $F$ a H\"older potential. Assume that $(\P,F)$ is of divergence type for a parabolic subgroup $\P$ of $\pi_1(M)$.  Then $\delta_\P^F<\delta_\Gamma^F$.
\end{theorem}
The next two results were also proved in \cite{dop} for $F=0$, but later generalized to include H\"older potentials (see  \cite[Theorem 8.3]{pps} and  \cite[Corollary 8.6]{pps} resp.).

\begin{theorem}\label{crigeofin2} Assume that $\delta_\Gamma^F$ is finite, and that $\Gamma$ is geometrically finite with $(\Gamma, F)$
of divergence type. Then the Gibbs measure $m_F$ is finite if and only if for every maximal  parabolic subgroup $\P$ we have that the sum 
$$\sum_{p\in \P} d(x,px)\exp\bigg(\int_x^{px} \widetilde{F}-\delta_\Gamma^F\bigg),$$
converges.
\end{theorem}

\begin{theorem}\label{crigeofin} Let $M$ be a geometrically finite manifold and $F$ a bounded H\"older potential. If $\delta_\P^F<\delta_\Gamma^F$ for every parabolic subgroup $\P$ of $\pi_1(M)$, then the Gibbs measure of $F$ is finite.
\end{theorem}

We finish this section with a definition that we will constantly use in this paper.  If $G$ is a group, we will use the notation $G^*$ to denote $G\setminus\{id\}$. 
\begin{definition}[Groups in Schottky position] Let $F_1$ and $F_2$ be discrete, torsion free subgroups of $Iso(\widetilde{M})$. We say that  $F_1$ and $F_2$ are in \emph{Schottky position} if there exist disjoint closed subsets $U_{F_1}$ and $U_{F_2}$ of $\partial_\infty \widetilde{M}$ such that $F_1^*(\partial_\infty \widetilde{M}\setminus U_{F_1})\subset U_{F_1}$ and $F_2^*(\partial_\infty \widetilde{M}\setminus U_{F_2})\subset U_{F_2}$.
\end{definition}

\subsection{Extended Schottky groups}\label{defesg}
 Let $N_{1} , N_{2} $ be two non-negative integers such that $N_1+N_{2}\geq 2$ and $N_2\ge 1$. Consider $N_{1}$ hyperbolic isometries $h_{1},...,h_{N_{1}}$ and $N_{2}$ parabolic ones $p_{1},...,p_{N_{2}}$ satisfying the following conditions:

\begin{enumerate}

\item  For $1\leq i\leq N_{1}$ there exists a compact neighbourhood $C_{h_{i}}$ of the attracting point $\xi_{h_{i}}$ of $h_{i}$ and a compact neighbourhood $C_{h_{i}^{-1}}$ of the repelling point $\xi_{h_{i}^{-}}$ of $h_{i}$, such that
    $$h_{i}(\partial \widetilde{M}\setminus C_{h_{i}^{-1}})\subset C_{h_{i}}.$$

\item  For $1\leq i\leq N_{2}$ there exists a compact neighbourhood $C_{p_{i}}$ of the unique fixed point $\xi_{p_{i}}$ of $p_{i}$, such that
$$\forall n\in\Z^{\ast} \quad p_{i}^{n}(\partial\widetilde{M}\setminus C_{p_{i}})\subset C_{p_{i}}.$$

\item  The $2N_{1}+N_{2}$ neighbourhoods introduced in $(1)$ and $(2)$ are pairwise disjoint.

\item The elementary parabolic groups $\langle p_{i}\rangle$, for $1\leq i\leq N_{2}$, are of divergence type.

\end{enumerate}

The group $\Gamma=<h_{1},...,h_{N_{1}},p_{1},...,p_{N_{2}}>$ is a non-elementary free group which acts properly discontinuously and freely on $M$ (see \cite[Corollary II.2]{dalpei}). Such a group $\Gamma$ is called an \emph{extended Schottky group}. It is proven in \cite{dalpei} that it is a geometrically finite group. Note that if $N_2=0$,  then the group $\Gamma$ only contains hyperbolic elements; in this case $\Gamma$ is a classical Schottky group and its geometric and dynamical properties are well understood.

\subsection{The Geometric potential}
We  briefly recall the construction of the geometric potential of $T^1M$. Pick a reference point $x_0\in \widetilde{M}$. For every $\xi\in \partial_\infty \widetilde{M}$ and $s\in \R$ we have a horoball $B_\xi(s)$ (see equation (\ref{horo})). The family $\{\partial B_\xi(s)\}_{s\in\R}$ foliates $\widetilde{M}$ by codimension one  hypersurfaces.  Let $x\in\widetilde{M}$ be the base of the vector $v\in  T^1M$ and $s_0\in \R$ such that $x\in\partial B_\xi(s_0)$. Each vector $\{\nabla_y b_\xi(o,y)\}_{y\in \partial B_\xi(s_0)}$ points to $\xi$ and is perpendicular to $\partial B_\xi(s_0)$. This defines a submanifold of $T^1\widetilde{M}$ passing through $v$, the so-called \emph{strong stable submanifold} at $v$; it will be denoted by $W^{ss}(v)$. A similar construction defines the \emph{strong unstable submanifold} at $v$, which we denote by $W^{su}(v)$. The strong (un)stable foliation is the foliation whose leaves are the strong (un)stable manifolds. One can characterize the points lying in $W^{ss}(v)$ and $W^{su}(v)$ by the following conditions:
$$W^{ss}(v)=\{w\in T^1\widetilde{M}: \lim_{t\to \infty} d(g_t v,g_tw)=0\}.$$
$$W^{su}(v)=\{w\in T^1\widetilde{M}: \lim_{t\to -\infty} d(g_t v,g_tw)=0\}.$$
For every $\gamma\in Iso(\widetilde{M})$  we have that $W^{ss}(\gamma v)=\gamma W^{ss}(v)$, and $W^{su}(\gamma v)=\gamma W^{su}(v)$. This implies that both foliations descend to $T^1M$. Observe that  $W^{ss}(g_t(v))=g_t (W^{ss}(v) )$, and $W^{su}(g_t(v))=g_t (W^{su}(v) )$; the geodesic flow preserves both foliations. 

\begin{definition}[Geometric potential] We define the \emph{geometric potential} or \emph{unstable jacobian} by the formula
$$F^{su}(\xi)=-\left.\dfrac{d}{dt}\right|_{t=0} \left.\log \det dg_t\right|_{W^{su}}(\xi),$$
where the determinant of $dg_t$ is computed with respect to an orthonormal basis (using the Riemannian metric $g$) of the unstable subspace at $\xi$. 
\end{definition}

The study of the geometric potential for (transitive) Anosov diffeomorphisms and flows is a classical subject (see \cite{si}, \cite{bob},  \cite{br}). In those cases the equilibrium state of the geometric potential corresponds to the SRB measure.  We will need the following result.

\begin{theorem}\cite[Theorem 7.2]{pps} \label{PPS1}
Let $\widetilde{M}$ be a simply connected pinched negatively curved manifold. Moreover assume that the derivatives of the sectional curvature are uniformly bounded. Then the geometric potential $F^{su}$ is H\"older continuous.
\end{theorem}

\noindent
It is proven in \cite[Theorem 3.9.1]{Kl} that if $(M,g)$ satisfies the pinching condition $-a^2\le K_g\le -b^2<0$, then we have \begin{align}\label{bounduj}-(N-1)a\le F^{su}(v)\le -(N-1)b,\end{align}
where $N$ is the real dimension of $M$ and $v\in T^1M$ (here $K_g$ stands for the sectional curvature of $(M,g)$). We emphasize that inequality (\ref{bounduj}) follows from a local computation: it is also true that if the sectional curvature of an open set $Z\subset M$ lies in $[-a_1^2,-b_1^2]$, then for every $v\in T^1Z$ we have \begin{align*}-(N-1)a_1\le F^{su}(v)\le -(N-1)b_1.\end{align*}
Inequality (\ref{bounduj}) also implies that for every invariant probability measure $\mu$ and $t\ge 0$  we have
$$h_\mu(g) -t(N-1)a\le h_\mu(g)+t\int F^{su}(v)d\mu(v)\le h_\mu(g) -t(N-1)b.$$
By the definition of the topological pressure it follows that
\begin{equation*} h_{top}(g)-t(N-1)a\le P(tF^{su})\le h_{top}(g)-t(N-1)b.  \end{equation*}

The following remark is a consequence of the local nature of inequality (\ref{bounduj}), and  inspires  the computations done in Section \ref{cons}.
\begin{remark}\label{insp} Define $U= F^{su}+(N-1)$. Assume that $$-(1+\frac{1}{N-1}M)^2\le K_g \le -(1+\frac{1}{N-1}L)^2,$$ on an open set $W$ of $M$. Then $$-(N-1)(1+\frac{1}{N-1}M) \le F^{su}(v)\le -(N-1)(1+\frac{1}{N-1}L),$$ for every vector $v\in T^1W$. Equivalently $-M \le U\le - L$.
\end{remark} 

It worth mentioning that under the assumptions of Theorem \ref{PPS1} Ruelle's inequality holds (see \cite{r}). More precisely we have that $$h_\mu(g)\le -\int F^{su}d\mu, $$ 
for every $\mu\in \M(g)$. In other words $P(F^{su})\le 0$. If $M$ has finite volume, then the equilibrium state of the geometric potential is the Liouville measure (see \cite{r}).

\section{Pressure map and SPR potentials}

In this section we will collect some useful information concerning the pressure map of potentials that vanish at infinity. We start with the definition of such potentials. 

\begin{definition}\label{C_0} We say that a continuous potential $F$ \emph{vanishes at infinity} if for every $\epsilon>0$, there exists a compact set $K\subset T^1M$ such that $\sup_{x\in K^c} |F(x)|<\epsilon.$ The space of continuous potentials vanishing at infinity is denoted by $C_0(T^1M)$.
\end{definition}

From now on we will focus on the geometrically finite case. The following two definitions will be very important in this paper. 
\begin{definition}[Topological entropy at infinity] Let $M$ be a geometrically finite manifold. We define the \emph{topological entropy at infinity} of the geodesic flow as $$\delta_\infty=\sup_\P \delta_\P,$$
where the supremum runs over the parabolic subgroups of $\pi_1(M)$. 
\end{definition}
\begin{definition}[SPR potentials]\label{sprdef} A potential $F\in C_0(T^1M)$ is called \emph{strongly positive recurrent} (SPR for short) if $P(F)>\delta_\infty$.
\end{definition}

The following result was recently obtained in \cite{rv}. 

\begin{theorem}\label{rv} Let $M$ be a geometrically finite manifold.  Let $(\mu_n)_n$ be a sequence of ergodic invariant probability measures converging to $\mu$ in the vague topology, and $F\in C_0(T^1M)$. Then 
$$\limsup_{n\to\infty} (h_{\mu_n}(g)+\int Fd\mu_n) \leq \|\mu\|(h_{\mu/|\mu|}(g)+\int Fd\mu/|\mu|)+(1-\|\mu\|)\delta_\infty.$$
\end{theorem}
We remark that if $\mu$ is the zero measure, then the right hand side is understood as $\delta_\infty$. A simple consequence of Theorem \ref{rv} is the following result.

\begin{corollary} \label{cor}Assume the hypothesis of Theorem \ref{rv}. Let $F\in C_0(T^1M)$ be a SPR potential. Then $F$ admits an equilibrium state.
\end{corollary}
\begin{proof}
Let $(\mu_n)_n$ be a sequence of ergodic probability measures such that $$h_{\mu_n}(g)+\int Fd\mu_n>P(F)-\frac{1}{n}.$$ We will assume that $(\mu_n)_n$ converges in the vague topology (otherwise take a subsequence). Let $\mu$ be the vague limit of the sequence $(\mu_n)_n$. By Theorem \ref{rv} we have that 
\begin{align*} P(F)=\limsup_{n\to \infty} \bigg(h_{\mu_n}(g)+\int Fd\mu_n\bigg)&\le |\mu|(h_{\mu/|\mu|}(g)+\int Fd\mu/|\mu|)+(1-|\mu|)\delta_\infty\\
&\le |\mu|P(F)+(1-|\mu|)\delta_\infty.
\end{align*}
Assume for a second that $\mu$ is not a probability measure, then we get 
$$ |\mu|P(F)+(1-|\mu|)\delta_\infty< P(F),$$
which leads to a contradiction. We conclude that $\mu$ is a probability measure. Finally  we obtain $$P(F)\le h_\mu(g)+\int Fd\mu,$$
which implies that $\mu$ is an equilibrium state of $F$. 
\end{proof}

We remark that for H\"older potentials this result also follows from a combination of Theorem \ref{gapdiv}, Theorem \ref{crigeofin} and  Lemma \ref{lem}. A more refined statement, whose proof is identical to the one of Corollary \ref{cor}, is our next result. 

\begin{theorem}\label{eeq2} Let $M$ be a geometrically finite manifold and  $F\in C_0(T^1M)$. Let $(\mu_n)_n$ be a sequence of ergodic invariant probability measures such that $$\lim_{n\to\infty}\big(h_{\mu_n}(g)+\int Fd\mu_n\big)=P(F).$$ Then the following statements hold. 
\begin{enumerate}
\item\label{i31} If $F$ is SPR, then  $(\mu_n)_n$  converges in the weak-* topology to an equilibrium state of $F$. 
\item\label{i32} Suppose that $F$ does not admit any equilibrium state.  Then  $(\mu_n)_{n}$  converges vaguely to the zero measure. In this case we  have  $P(F)=\delta_\infty$. 
\item\label{i33}  Suppose that $F$ does admit an equilibrium state.  Then the accumulation points of  $(\mu_n)_{n}$  lies in the set $$\{t\mu:t\in [0,1]\text{ and }\mu\text{ is an equilibrium state of $F$}\}.$$
\end{enumerate}
\end{theorem}

We remark that if $F$ is H\"older continuous, then $F$ has at most one equilibrium state (see Theorem \ref{pps}). In particular if $F\in C_0(T^1M)$ is SPR and H\"older, then any sequence of ergodic probability measures satisfying $$\lim_{n\to\infty}\big(h_{\mu_n}(g)+\int Fd\mu_n\big)=P(F),$$
must converge to the unique equilibrium state. Another important consequence of the SPR property is the following result (compare with \cite[Proposition 5.8]{rv}). 
\begin{theorem}[First derivative of the pressure]\label{deriva} Let $M$ be a geometrically finite manifold and $F\in C_0(T^1M)$ a SPR H\"older potential. For every $G\in C_b(T^1M)$ the following holds
$$\frac{d}{dt}_{|t=0}P(F+tG)=\int Gd\mu_F,$$
where $\mu_F$ is the equilibrium state of $F$. 
\end{theorem} 
\begin{proof} Let $\mu_0$ be the equilibrium state of $F$ and for every $t\ne 0$ we choose an ergodic probability measure $\mu_t$ such that $$h_{\mu_t}(g)+\int (F+tG)d\mu_t\ge P(F+tG)-t^2.$$
 Observe that 
\begin{align*} P(F+tG)-P(F)&\le \big(h_{\mu_{t}}(g)+\int (F+tG)d\mu_{t}+t^2\big)-\big(h_{\mu_{t}}(g)+\int F d\mu_{t}\big)\\
&=t\int Gd\mu_{t}+t^2.
\end{align*}
Similarly 
\begin{align*} P(F+tG)-P(F)&\ge \big(h_{\mu_{0}}(g)+\int (F+tG)d\mu_{0}\big)-\big(h_{\mu_{0}}(g)+\int Fd\mu_{0}\big)\\
&=t\int Gd\mu_0.
\end{align*}
In particular for $t>0$ we get 
$$\int Gd\mu_0\le \frac{P(F+tG)-P(F)}{t}\le \int Gd\mu_t+t,$$
and the reversed inequality for $t<0$. We now claim that $(\mu_t)_t$ converges in the weak-* topology to $\mu_0$ as $t$ goes to zero. First observe that 
$$P(F+tG)-t^2\le h_{\mu_t}(g)+\int (F+tG)d\mu_t\le P(F+tG),$$
therefore $$\lim_{t\to 0}\bigg(h_{\mu_t}(g)+\int F d\mu_t\bigg)=\lim_{t\to 0}\bigg(h_{\mu_t}(g)+\int (F+tG)d\mu_t\bigg)=\lim_{t\to 0}P(F+tG)=P(F).$$
Since $F$ is H\"older we know that $\mu_0$ is the unique equilibrium state of $F$. We now use Theorem \ref{eeq2} to conclude that $(\mu_t)_t$ convergues in the weak-* topology to $\mu_0$ as $t$ goes to zero. As a consequence we obtain that $\lim_{t\to 0}\int Gd\mu_t=\int Gd\mu_0$. This together with the inequalities above give us that $$\frac{d}{dt}_{|t=0}P(F+tG)=\int Gd\mu_0.$$
\end{proof}

The following two simple facts will be constanly used in this paper. For completeness we provide their proofs (see \cite{rv}). 
\begin{lemma} \label{lem} Let $M$ be a geometrically finite manifold and $F\in C_0(T^1M)$ a H\"older continuous potential. Then for every maximal parabolic subgroup $\P$ of $\pi_1(M)$ we have $\delta^{F}_\P=\delta_\P$.
\end{lemma}
\begin{proof} Pick a reference point $x\in M$ that belongs to the region of the manifold where the cusp is standard, i.e. where it looks like the quotient of a horoball by the maximal parabolic subgroup $\P$. We will moreover assume that $x$ belongs to the region of the cusp where $|F|<\epsilon$.  Using the convexity of the horoballs we get that the Poincar\'e series of $\P$ based at $x$ can be bounded below by $\sum_{p\in\P}\exp{(-(s+\epsilon)d(x,px))}$, and above by $\sum_{p\in\P}\exp{(-(s-\epsilon)d(x,px))}$. This implies that $|\delta_\P^F-\delta_\P|<2\epsilon$. Since $\epsilon$ was arbitrary we conclude the lemma. We remark that since $F$ is H\"older continuous the behaviour of the Poincar\'e series is independent of the base point. 
\end{proof}

\begin{lemma}\label{lem22} Let $M$ be a geometrically finite manifold and $F\in C_0(T^1M)$ a H\"older continuous potential. Then $P(F)\ge \delta_\infty$.
\end{lemma}
\begin{proof} Lemma \ref{lem} implies that $\delta_\P=\delta_\P^F$, for every maximal parabolic subgroup $\P$ of $\pi_1(M)=\Gamma$. It follows from the definition of the Poincar\'e series and the inclusion $\P\subset \Gamma$ that $\delta_\P^F\le \delta_\Gamma^F$. Since $F$ is H\"older continuous we know that $P(F)=\delta_\Gamma^F$ (see Theorem \ref{pps}). Finally we conclude that $$\delta_\infty=\sup_\P\delta_\P=\sup_\P \delta_\P^F\le \delta_\Gamma^F=P(F).$$
\end{proof}

The following result was obtained in \cite[Theorem 5.7]{rv}.

\begin{theorem}[Pressure map for positive potentials in $C_0(T^1M)$]\label{description} Let $M$ be a geometrically finite manifold. Then the pressure map of a positive H\"older potential $F\in C_0(T^1M)$ verifies the following properties:
\begin{enumerate}
\item[(1)] for every $t \in \R$ we have that $P(t F) \geq \delta_\infty$
\item[(2)] the function $t\mapsto P(tF)$ has a horizontal asymptote at $-\infty$, that is
$$ \lim_{t \to -\infty} P(tF)= \delta_\infty.$$
\end{enumerate}
Moreover, if $t_F:= \sup \left\{ t \leq 0 : P(tF)=\delta_\infty \right\}$, then
\begin{enumerate}
\item[(3)] for every $t>t_F$ the potential $tF$ has an equilibrium state, and
\item[(4)] the pressure function $t\mapsto P(tF)$ is differentiable in $(t_F,\infty)$, and it verifies
\begin{equation*}
P(tF)=
\begin{cases}
\delta_\infty & \text{ if } t < t_F\\
\text{strictly increasing}  & \text{ if } t > t_F,
\end{cases}
\end{equation*}
\item[(5)] If $t<t_F$ then the potential $tF$ does not have an equilibrium state.
\end{enumerate}
\end{theorem}

In this paper we will deal with potentials that are not necessarily positive. If the potential $F$ is non-negative, then we have a similar  description of the pressure map. The only difference is that the limit $$\lim_{t\to-\infty}P(tF)=A,$$
does not need to be equal to $\delta_\infty$ (by Lemma \ref{lem22} we know that $A\ge \delta_\infty$). We finish this section with the definition of phase transition that we will use in this paper. 
\begin{definition}[Phase transition for the geodesic flow]\label{phdef} We say that a potential $F$ exhibits \emph{a phase transition at }$t_0$ if there exists $\epsilon>0$ such that $P(tF)$ has an equilibrium state for $t\in (t_0,t_0+\epsilon)$, but it does not have for $t\in (t_0-\epsilon,t_0)$ (or vice versa). A potential $F$ exhibits a phase transition if it exhibits a phase transition for some $t_0\in \R$.  \end{definition}

As mentioned in the introduction our best result concerning the regularity of the pressure map is Theorem \ref{deriva}, so we can not define phase transitions in terms of points where the pressure map is not real analytic (as in the symbolic case).

\section{Phase transitions}\label{pha}
In this section we will prove Theorem \ref{i1}. We start with a  modification of  \cite[Theorem C]{dop}. 

\begin{proposition} \label{exp} Suppose we have a parabolic subgroup $\P$ and a hyperbolic isometry $h$ such that $\P$  and $H=\langle h\rangle$ are in Schottky position. Denote by $\Gamma_k$  the group generated by $\P$ and $\langle h^k\rangle$, and define $M_k=\widetilde{M}/\Gamma_k$. Let $F:T^1M_1\to \R$ be a  bounded H\"older potential for which $(\P,F)$ is of convergence type and $\delta_\P^F>0$.  Suppose that $\int_\gamma F=0$, where $\gamma$ is the periodic orbit associated to $h$. Denote by $F_n$ the lift of $F$ to $T^1M_n$. Then there exists $N_0$ such that if $n\ge N_0$, then $(\Gamma_n,F_n)$ is of convergence type, and $\delta_\P^F=\delta_{\Gamma_n}^{F_n}$.
\end{proposition}
\noindent
The heart of the proof of \cite[Theorem C]{dop} is the following elementary fact about Hadamard manifolds. Given $D>0$, there exists $C=C(D)>0$ such that for every geodesic triangle with vertices $x,y,z$, and  angle at $z$ bigger than $D$, then $$d(x,y)\ge d(x,z)+d(z,y)-C.$$ We begin with an analogous inequality that the Gibbs cocycle $C_{F,\xi}(x,y)$ satisfies.

\begin{lemma}\cite[Lemma 3.4]{pps}\label{1} Let $F$ be a H\"older potential such that  $||F||_0<B$. Then for every $r>0$, for all  $x,y\in\widetilde{M}$ and $\xi\in \mathcal{O}_xB(y,r)$ we have 
$$|C_{F,\xi}(x,y)+\int_x^y\widetilde{F}|\le L(r,B),$$
for certain uniform constant $L(r,B)$. Here $ \mathcal{O}_xB(y,r)\subset \partial_\infty \widetilde{M}$ denotes the set of end points of geodesic rays emanating from $x$ that intersect $B(y,r)$.
\end{lemma}

In the inequality above we are using the Gibbs cocycle defined in Section \ref{pre}. In the proof of Proposition \ref{exp} we will use that $C_{F,\xi}(x,y)$ is an additive cocycle, more precisely, that $$C_{F,\xi}(x,y)+C_{F,\xi}(y,z)=C_{F,\xi}(x,z).$$ 

\begin{proof}[Proof of Proposition \ref{exp}]
Since $H$ and $\P$ are in Schottky position we can find $U_\P,U_H\subset \widetilde{M}\cup \partial_\infty \widetilde{M}$ so that 
\begin{enumerate}
\item\label{100} $\P^*(\partial_\infty \widetilde{M}\setminus U_{\P})\subset U_{\P}$.
\item\label{200}$H^*(\partial_\infty \widetilde{M}\setminus U_{H})\subset U_{H}$.
\item $U_{H}\cap U_{\P}=\emptyset$.
\end{enumerate}
Fix $x\in M$ over the axis of $h$ so that  $x\not\in U_{H}\cup U_{\P}$. As a consequence of the Ping Pong Lemma we have that $\Gamma$ is isomorphic to the free product $H*\P$.  By the comments above Lemma \ref{1} and the fact that $U_H\cap U_\P=\emptyset$, we know that there exists a positive constant $C$ such that for every $y\in U_{H}$ and $z\in U_{\P}$ we have
\begin{align} \label{eq:tri}
d(y,z)\ge d(x,y)+d(x,z)-C.
\end{align}
 Applying inequality  \eqref{eq:tri} and the inclusions described in (\ref{100}) and (\ref{200}) we obtain
\begin{equation*} \label{eq:tria}
d(x,p_1h^{kn_1}...p_j h^{kn_j}x)\ge \sum_i d(x,p_i x)+ \sum_i d(x,h^{kn_i}x)-2jC,
\end{equation*}
where $n_i\in \Z^*$, $k\ne 0$ and $p_i\in\P^*$. Let $B$ be a bound for $F$, i.e. $||F||_0<B$. Choosing $\xi$ outside $U_\P \cup U_H$ and $r$ big enough we can apply Lemma \ref{1}. Then
\begin{align*}\int_x^{h^nx}\widetilde{F}+\int_x^{p  x}\widetilde{F}= \int_{px}^{p h^nx}\widetilde{F}+\int_x^{p x}\widetilde{F}  &\ge -2L(r,B)-C_{F,\xi}(x,p h^nx)\\
&\ge -3L(r,B)+\int_x^{p h^nx}\widetilde{F}.
\end{align*}
This immediately generalize to  
\begin{align*}\int_x^{p_1h^{kn_1}...p_jh^{kn_j}x}\widetilde{F} \le (2j+1)L(r,B) +
\sum_i \int_x^{h^{kn_i}x}\widetilde{F}+\sum_i \int_x^{p_i  x}\widetilde{F} .
\end{align*}
Define $l=d(x,hx)$. By the choice of $x$ we have that $d(x,h^{N}x)=|N| l$, and since  $\int_x^{hx}\widetilde{F}=0$ we also have  $\int_x^{h^Nx}\widetilde{F}=0$. Finally
$$\sum_{n\in\Z^*}\exp\bigg(\int_x^{h^{kn}x} (\widetilde{F}-s)\bigg) =\sum_{n\in\Z^*}\exp(-s|n|lk) =2\dfrac{\exp(-slk)}{1-\exp(-slk)}.$$
For simplicity we will bound the expression 
$$\tilde{P}(s)=\sum_{j\ge 1}\sum_{p_i\in\P^*,m_i\in\Z^*}\exp(\int_x^{p_1h^{kn_1}...p_jh^{kn_j}x}(\widetilde{F}-s)).$$
A bound for $P(s)$ follows identically, but here we have more symmetry. Using the inequalities above we obtain
 \begin{equation}\label{12}
 \tilde{P}(s)\le \sum_{j\ge 1} \left(e^{2(C+L(r,B)+1)} \sum_{n\in\Z^*}\exp\bigg(\int_x^{h^{kn} x} (\widetilde{F}-s)\bigg)   \sum_{p\in\P^*}\exp\bigg(\int_x^{p x} (\widetilde{F}-s)\bigg)  \right)^j .
 \end{equation}

Taking $k$ big enough we can ensure the right hand side of (\ref{12}) to be convergent at $s=\delta_\P^F$. In particular $\delta_\P^F\ge \delta_{\Gamma_k}^{F_k}$, which immediately implies $\delta_{\Gamma_k}^{F_k}=\delta_\P^{F}$. Inequality (\ref{12}) also implies the convergence property of the pair $(\Gamma_k,F_k)$. 
\end{proof}

\begin{remark}\label{generalized} Let $\{h_1,...,h_l\}$ be a collection of hyperbolic isometries and denote by $H_i$ the group generated by $h_i$. Suppose that the subgroups $\{H_1,...,H_l,\P\}$ are pairwise in Schottky position (as in the definition of extended Schottky, but allowing $\P$ to have bigger rank). Moreover, assume that the integral of $F$ over the closed geodesic associated to each $h_i$ vanishes. Define $\Gamma_n$ as the group generated by $\P$ and the elements $\{h_1^n,...,h_l^n\}$ . The proof of Proposition \ref{exp} can be modified to conclude that  for big enough $k$ we have $$\delta_\P^F=\delta_{\Gamma_k}^{F_k},$$ and that $(\Gamma_k,F_k)$ is of convergence type. For simplicity we will state our results only for the case treated in Proposition \ref{exp}, but we emphasize that everything works identically under the hypothesis of this remark. 
\end{remark}

The next proposition is very important for us, it provides the family of potentials for which we will have  phase transitions. In Section \ref{cons} we will use ideas from the proof of Proposition \ref{pot} to modify the metric at the cusp of a hyperbolic manifold, in order to achieve phase transitions for the geometric potential. Recall that $(\P,-F)$ is of convergence type if the sum 
$$\sum_{p\in\P}\exp\bigg(-\int_x^{px}\widetilde{F}-\delta^{-F}_\P d(x,px)\bigg),$$
is finite. 

\begin{proposition} \label{pot}  Let $M$ be a geometrically finite manifold. There exists a H\"older potential $F_0$ satisfying the following properties:
\begin{enumerate}
\item $F_0\in C_0(T^1M)$,
\item $F_0$ is positive near the cusps,
\item $(\P,-F_0)$ is of convergence type for every maximal parabolic subgroup $\P$ of $\pi_1(M)$.
\end{enumerate}

\end{proposition}

\begin{proof}  We will  define $F_0$ in a neighborhood of each cusp, and then we will extend $F_0$ to the rest of the manifold in a H\"older continuous way (making sure that $F_0\in C_0(T^1M)$). Pick a maximal parabolic subgroup $\P$ of $\pi_1(M)$, and denote by $\xi\in \partial_\infty \widetilde{M}$ to its fixed point. There exists a neighborhood $\mathcal{U}$ of the cusp associated to $\P$ which is isometric to $B_\xi(q_0)/\P$, for big enough $q_0$. Recall taht $\pi:T^1M\to M$ is the canonical projection.  Pick  a reference point $x\in \partial \cU$. For $w\in\mathcal{U}$ we define $\widehat{d}(w)=q$, if $\pi(w)\in \partial B_\xi(q+q_0)/\P$. We say that the geodesic $\gamma:[a,b]\to T^1M$  has \emph{height} $H$ if $$\max_{t\in[a,b]}\widehat{d}(\gamma(t))=H.$$ For $l<L$, define  $$S(l,L)=\{p\in\P:l<d(x,px)\le L\}.$$
By the definition of critical exponent, for every $\epsilon>0$ there exists a real number $C(\epsilon)$ so that
$$\sum_{p\in S(C(\epsilon),\infty)} \exp(-(\delta_\P +\epsilon) d(x,px))<\epsilon^2.$$
We define a sequence $(A_n)_n$ inductively as follows: let $A_1=C(1)$, and $A_{n+1}=\max(A_n+1,C(1/n))$. By construction the sequence of real numbers $(A_n)_{n}$ is strictly increasing and satisfies 
$$\sum_{p\in S(A_n,\infty)} \exp(-(\delta_\P +1/n)d(x,px))<\dfrac{1}{n^2}.$$
We define $H_n$ as the maximum height of the geodesic segments $[x,px]$, where $p$ runs in $S(A_n,A_{n+1})$. With the heights $(H_n)_n$ we construct a sequence $(B_n)_{n}$ by declaring $B_1=H_1$, and inductively define $B_{n+1}=\max(B_n+1, H_{n+1})$. Define a function $f$ on $\mathcal{V}=\widehat{d}^{-1}([B_1,\infty))\subset \cU$, by the following expression
\begin{equation*}
f(x)=
\begin{cases}
-\widehat{d}(x)+1/n-B_n  & \text{ if } x\in \widehat{d}^{-1}([B_n,B_n+1/n-1/(n+1)]) \\
1/(n+1)  & \text{ if } x\in  \widehat{d}^{-1}([B_n+1/n-1/(n+1), B_{n+1}]).
\end{cases}
\end{equation*}

Let $F:T^1\mathcal{V}\to \R$, be the composition of the projection from $T^1\mathcal{V}$ to $\mathcal{V}$ and $f$. We do the same construction for each cusp in $M$. Using these functions at the cusps and  any H\"older continuous extension to the rest of the manifold (making sure that $F_0\in C_0(T^1M)$), we obtain our H\"older continuous potential $F_0$.  We will now check that $F_0$ satisfies the properties described in Proposition \ref{pot}.
It follows from the construction of $f$ that $F_0$ goes to zero through the cusps. By  Lemma \ref{lem} to get that $\delta_\P^{-F_0}=\delta_\P$. It only remains to check that $(\P,-F_0)$ is of convergence type for every maximal parabolic subgroup $\P$ of $\pi_1(M)$. Observe that if $p\in S(A_n,A_{n+1})$, then $[x,px]$ has height at most $H_n$, in particular at most $B_n$. Because of the way we defined the function $f$ we obtain that if $p\in S(A_n,A_{n+1})$ and $v\in T^1[x,px]$, then $F_0(v)\ge\frac{1}{n}$.  In other words, for $p\in S(A_n,A_{n+1})$ we have  $$\int_x^{px}F_0\ge \frac{d(x,px)}{n}.$$ Finally\\
$$\sum_{p\in S(A_1,\infty)} \exp\bigg(-\int_x^{px}F_0-\delta_\P d(x,px)\bigg)=\sum_{n=1}^\infty \sum_{p\in S(A_n,A_{n+1})} \exp\bigg(-\int_x^{px}F_0-\delta_\P d(x,px)\bigg)$$

\begin{align*}
&\le \sum_{n=1}^\infty \sum_{p\in S(A_n,A_{n+1})} \exp\bigg(-\frac{1}{n}d(x,px)-\delta_\P d(x,px)\bigg) \\
&\le \sum_{n=1}^\infty \sum_{p\in S(A_n,\infty)} \exp\bigg(-\bigg(\frac{1}{n}+\delta_\P\bigg)d(x,px)\bigg) \\
&\le \sum_{n=1}^\infty \frac{1}{n^2},
\end{align*}
which is finite.

\end{proof}

\begin{remark}\label{rem}
Since the numbers $(B_n)_n$ can be arbitrarity far apart, we interpret the decay of $F$ through the cusp associated to $\P$ as `very slow'.
\end{remark}

\begin{definition}\label{class} A potential $G$ belongs to the family $\F_s$ if the following conditions are satisfied.
\begin{enumerate}
\item $G\in C_0(T^1M)$,
\item $G$ is positive in a neighborhood of the cusps of $M$,
\item $(\P,-G)$ is of convergence type for every maximal parabolic subgroup $\P$ of $\pi_1(M)$.
\end{enumerate}
The elements in $\F_s$ are called potentials \emph{going slowly to zero} through the cusps of $M$. The class of non-negative potentials in $\F_s$ is denoted by $\F_s^+$.
\end{definition}

The family $\F_s$ is not empty because of Lemma \ref{lem}. We remark that the family $\F_s$ is quite big: if $F\in \F_s$ and $G$ is a H\"older potential in $C_0(T^1M)$ satisfying $G\ge F$ in a neighborhood of the cusps, then $G\in \F_s$. We now proceed to prove Theorem \ref{i1}. The statement of our next result is  more involved that Theorem \ref{i1}, but it has the advantage of being more precise. Our manifolds are geometrically finite with one cusp, in particular $\delta_\infty=\delta_\P$. As in Theorem \ref{description} we use the notation $$t_{F}=\sup\{t: P(tF)=\delta_\P\}.$$

\begin{theorem} \label{teo6} Let $\P$ be a divergence type parabolic subgroup and $h$ a hyperbolic isometry such that $\langle h\rangle$ and $\P$ are in Schottky position. Define $\Gamma_n$ as the group generated by $\P$ and $\langle h^n\rangle$. Let $M_k=\widetilde{M}/\Gamma_k$, and $F_1:T^1M_1 \to \R$  a potential in the class $\F_s^+$. Assume that $\int_\gamma F_1=0$, where $\gamma$ is the periodic orbit associated to $h$. Let $F_n$ be the lift of $F_1$ to $T^1M_n$, and $N_0$ be the constant provided by Proposition \ref{exp} for $F=-F_1$. Then for $n\ge N_0$  we have:
\begin{enumerate}
\item $t_{F_n}\in [-1,0)$.
\item The potential $tF_n$ has an equilibrium state for $t>t_{F_n}$.
\item The potential $tF_n$ does not have an  equilibrium state for $t<t_{F_n}$.
\end{enumerate}
In conclusion, the pressure map of $F_n$ exhibits a phase transition at $t=t_{F_n}$. Moreover, the pressure map is differentiable in $(-\infty, t_{F_n})\cup (t_{F_n},\infty)$. With respect to the behaviour at $t=t_{F_n}$ we have two possibilities:
\begin{enumerate}
\item[(4)] If the potential $t_{F_n}F_n$ does not have an equilibrium state, then the pressure map is differentiable everywhere.
\item[(5)]  If $t_{F_n}F_n$ has an equilibrium state, then the pressure map is not differentiable at $t=t_{F_n}$. 
\end{enumerate}

\end{theorem}

We remark that since $\P$ is of divergence type, then the geodesic flow on $M_n$ has a measure of maximal entropy (a consequence of Theorem \ref{gapdiv} and Theorem \ref{crigeofin}).

\begin{center}
\begin{tikzpicture}\label{fig2}[scale=1]

%linea intermitente
\draw[color=gray!50] [dashed](-2,1)--(2,1);

%linea curva
\draw[red] plot[smooth] coordinates {(-2,1) (0,2.5) (1,4)};

%Ejes X
\draw [>= stealth, ->](-5,0)--(2,0) node[below=12pt,midway]{Figure 1: Phase transition for $ F_1\in\F_s^+$};
\draw (2,0) node[below]{$t$};

%Ejes Y
\draw [>= stealth, ->](0,-0.2)--(0,4);
\draw (0,4) node[left]{$P(tF_n)$};

\draw (0,2.5) node{$\bullet$} node[right]{$\delta_{\Gamma_n}$};

\draw (0,1) node{$\bullet$} node[above right]{$\delta_\P$};

\draw (-3,0) node{$\bullet$} node[below]{$-1$};

%linea recta
\draw [domain=-5:-2]plot(\x,1);

\draw (-2,1) node{$\bullet$} node[below]{$t_{F_n}$};

\end{tikzpicture}
\end{center}

\begin{proof} %First observe that by definition of the potentials $(F_n)_n$, they lift to the same potential $\widetilde{F}$ on $T^1\widetilde{M}$.  
Since the covering map $M_k\to M_1$, is one to one in a neighborhood of the cusp associated to $\P$ we get that the potential $F_n$ belongs to $C_0(T^1M_n)$. This implies that $\delta_\P=\delta_{\P}^{tF_n}$ (see Lemma \ref{lem}). Since $F_n$ is non-negative, the pressure map $t\mapsto P(tF_n)$ is non-decreasing. It follows from Corollary \ref{cor} that for $t>t_{F_n}$  there exists an equilibrium state for $tF_n$, and from Theorem \ref{deriva} that the pressure map is differentiable in $(t_{F_n},\infty)$.   By assumption $\P$ is of divergence type, then  Theorem \ref{gapdiv} gives us   $\delta_\infty=\delta_\P<\delta_{\Gamma_n}=P(0)$. We conclude that $t_{F_n}<0$. Since $F_1\in \F_s^+$ we know that $(\P,-F_1)$ is of convergence type. Notice that each $F_n$ lifts to the same potential on $T^1\widetilde{M}$, in particular the Poincar\'e series of $(\P,-F_n)$ is independient of $n$. It follows that $(\P,-F_n)$ is of convergence type for every $n\ge 1$. By the choice of $N_0$ (see Proposition \ref{exp}) we know that $\delta_\P^{-F_n}=P(-F_n)$, for every $n\ge N_0$. We can conclude that $\delta_\infty=P(-F_n)$; this immediately implies that $t_{F_n}\ge -1$. Moreover, by Hopf-Tsuji-Sullivan-Roblin theorem there is not equilibrium state for the potential $-F_n$ (since $(\Gamma_n, -F_n)$ is of convergence type; see \cite[Theorem 1.8]{pps}). Since $P(tF)\ge \delta_\infty$ (see Lemma \ref{lem22}) we can conclude that the pressure map is constant in $(-\infty, t_F)$. We will now check part (3). Suppose there exists $t\in (-\infty,t_{F_n})$ such that $tF_n$ has an equilibrium state, say $\mu_t$. Recall that by definition of the family $\F_s^+$, the potential $F_n$ is non-negative and positive in a neighborhood of the cusp associated to $\P$. Since the support of $\mu_t$ contains some portion of the cusp we conclude that $\int F_nd\mu_t>0$. Observe that for every $t'>t$ we have  $$\delta_\infty=h_{\mu_t}(g)+t\int F_nd\mu_t<h_{\mu_t}(g)+t'\int F_nd\mu_t.$$ This implies that for $t'\in (t,t_{F_n})$ we have $P(t'F_n)>\delta_\infty$, which is a contradiction. We conclude that the potential $tF_n$ does not have an equilibrium state for $t<t_{F_n}$. This proves part (3). We now proceed to prove part (4) and part (5). First suppose that $t_nF_n$ does not have an  equilibrium state. Let $\mu_t$ be the equilibrium state of $tF_n$, for $t>t_{F_n}$. Observe that 
\begin{align*}\lim_{t\to t_{F_n}^+} \bigg(h_{\mu_t}(g)+t_{F_n}\int F_nd\mu_t\bigg)&=\lim_{t\to t_{F_n}^+} \bigg(h_{\mu_t}(g)+t\int F_nd\mu_t\bigg)\\
&=\lim_{t\to t_{F_n}^+}P(tF_n)=P(t_nF_n).\end{align*}
By Theorem \ref{eeq2} we conclude that  $(\mu_t)_t$ converges vaguely to the zero measure as $t$ goes to $t_{F_n}$. In particular $$\lim_{t\to t_{F_n}^+}\int F_nd\mu_t=0.$$
By Theorem \ref{deriva} we know that the derivative of the pressure map of $F_n$ at $t$, for $t>t_{F_n}$, is given by $\int F_nd\mu_t$. We conclude that the pressure map is differentiable at $t=t_{F_n}$, and that the derivative at that point is equal to zero. This proves part (4). Now assume that $t_nF_n$ has an equilibrium state, and denote it by $\mu_{t_{F_n}}$. As in the proof of part (4) we will denote by $\mu_t$ the equilibrium state of $tF_n$, for $t>t_{F_n}$. The same argument used in the proof of part (4) allows us to conclude that $\int F_nd\mu_{t_{F_n}}>0$.  The left hand side derivative of the pressure map at $t=t_{F_n}$ is clearly equal to zero. We will prove that the right hand side derivative of the pressure map at $t=t_n$ is equal to $\int F_nd\mu_{t_{F_n}}$. As observed in part (3) we know that $$\lim_{t\to t_{F_n}^+}\bigg(h_{\mu_t}(g)+t_{F_n}\int F_nd\mu_t\bigg)=P(t_nF_n).$$
Theorem \ref{eeq2} allows us to conclude that every vague limit point of $(\mu_t)_t$ as $t$ goes to $t_{F_n}$, must be of the form $\lambda \mu_{t_{F_n}}$ (for some $\lambda\in [0,1]$). We now claim that $(\mu_t)_t$ actually converges to $\lambda\mu_{t_{F_n}}$ (for a fixed $\lambda$). The convexity and differentiability of the pressure map in $(t_{F_n},\infty)$ implies that the limit $$\lim_{t\to t_{F_n}^+} \int F_nd\mu_t=A,$$ exists. If $\lambda\mu_{t_{F_n}}$ is a limit point of $(\mu_t)_t$ as $t$ goes to $t_{F_n}$, then we must have $A=\lambda\int F_n d\mu_{t_{F_n}}$. We conclude that there is at most one possible choice for $\lambda$, and that the sequence of measures is convergent. Observe that the right hand side derivative of the pressure at $t=t_{F_n}$ is equal to $A=\lambda\int F_n d\mu_{t_{F_n}}$. In particular, it is less than or equal to $\int F_n d\mu_{t_{F_n}}$. Define $$L(t)=h_{\mu_{t_{F_n}}}(g)+t\int F_nd\mu_{t_{F_n}}.$$ Observe that $P(tF)\ge L(t)$, and that $P(t_nF)=L(t_n)$. The convexity of the pressure map implies that the right hand side derivative at $t=t_{F_n}$ is at least $\int F_n d\mu_{t_{F_n}}$. We conclude that the right hand side derivative of the pressure at $t=t_{F_n}$ must be equal to $\int F_n d\mu_{t_{F_n}}$. We remark that since $A=\int F_n d\mu_{t_{F_n}}$, we necessarily have $\lambda=1$. In other words, the sequence $(\mu_t)_t$ converges to $\mu_{t_{F_n}}$ as $t$ goes to $t_{F_n}$. 

\end{proof}

\begin{remark} Suppose the potential $F$ satisfies the following additional property: $(\P,-F)$ is of convergence type and $(\P,tF)$ is of divergence type for $t>-1$. In this case we can ensure that for every $t> -1$ the potential $tF_n$ is SPR; this implies that  $t_{F_n}=-1$. Moreover, at $t=-1$ the potential $tF_n$ does not have an  equilibrium state. By part (4) of Theorem \ref{teo6} we get that the pressure map is differentiable everywhere. 
\end{remark}

\begin{remark} If the potential $t_nF_n$ has an equilibrium state, then the pressure map exhibits a first order phase transition (using Ehrenfest classification). In this case the first derivative of the pressure develops a singularity. If the potential $t_nF_n$ does not have an equilibrium state, then it is reasonable to expect that some higher order derivative of the pressure map should develop a singularity. %For shifts of finite type, the second derivative of the pressure map (of a H\"older potential) has a strong connection to the central limit theorem (see \cite{ru}).
 In the context of countable Markov shifts, Sarig  investigated the relation between critical exponents and abnormal fluctuations (see \cite{s3}). It is a very interesting project to try to prove analogous result to those in \cite{s3} for the geodesic flow on a pinched negatively curved manifold.  

\end{remark}

We will briefly discuss what happen if we only assume $F\in \F_s$, i.e. we allow $F$ to take negative values.  As before we will assume that $M$ has only one cusp. First suppose that for every $\mu\in \M(g)$ we have $\int F_nd\mu\ge 0$. Then Theorem \ref{deriva} and the convexity of the pressure map implies that $t\mapsto P(tF_n)$, has the same description as the one of a potential in $\F_s^+$. As explain in the paragraph after Theorem  \ref{description} we need to be a bit careful here, in general we can not immediately conclude that $\lim_{t\to -\infty}P(tF)=\delta_\infty$. In our case this issue does not apply: our construction ensures that for some $t$ we have $P(tF)=\delta_\infty$. If there exists   $\mu\in \M(g)$ such that $\int F_nd\mu< 0$, then the set $$J=\{t\in \R: P(tF_n)=\delta_\infty\},$$ is a compact interval. Observe that for $t\in \R\setminus J$ the potential $tF_n$ is strongly positive recurrent and therefore admits an equilibrium state. We claim that if $t\in int(J)$, then  $tF_n$ does not have an equilibrium state. Suppose that for some $t\in int(J)$ there exists an equilibrium state for $tF$, say $\mu_t$. If $\int F_nd\mu_t>0$, then for $t'>t$ we have 
$$\delta_\infty=h_{\mu_t}(g)+t\int F_n d\mu_t<h_{\mu_t}(g)+t'\int F_n d\mu_t.$$
In particular if $t'\in J$ and $t'>t$, then $P(t'F)>\delta_\infty$. Similarly, if $\int F_nd\mu_t<0$, then for $t'<t$ we have $$\delta_\infty=h_{\mu_t}(g)+t\int F_n d\mu_t<h_{\mu_t}(g)+t'\int F_n d\mu_t.$$
In particular if $t'\in J$ and $t'<t$, then $P(t'F)>\delta_\infty$. We conclude that $\int F_nd\mu_t=0$ (otherwise we contradict the definition of $J$). Since $\mu_t$ is an equilibrium state for $tF_n$ we have $$\delta_\infty=P(tF_n)=h_{\mu_t}(g)+t\int F_nd\mu_t=h_{\mu_t}(g).$$
Observe that $\int F_nd\mu_t=0$ and $h_{\mu_t}(g)=\delta_\infty$ implies that $\mu_t$ is an equilibrium state for $sF_n$, for every $s\in J$. Since by construction $-F_n$ does not have an equilibrium state, and $-1\in J$, we conclude that the measure $\mu_t$ can not exist. We remark that potentials with this description of the pressure map (see Figure 2) can be constructed in a similar fashion to those in Theorem \ref{teo6}.  For instance suppose that $F_1$ is negative in the complement of a neighborhood of the cusp and identically zero on $\gamma$ (the geodesic associated to the hyperbolic generator $h$). If we take a closed geodesic that wrap around the lift of $\gamma$ to $M_n$ a large number of times (the geodesic represented by $ph^{nk}$ for big $k$ works) we get an invariant measure with negative integral against $F_n$. As mentioned in the paragraph below Definition \ref{class}, if $F\in \F_s$ and $t>1$, then $tF\in \F_s$. In particular, given $M>1$, there exists $N_1=N_1(M,F)$ such that the following holds: for every $n\ge N_1$, the pairs $(\Gamma_n,-MF_n)$ and  $(\Gamma_n,-F_n)$ are of convergence type and have critical exponent $\delta_\infty=\delta_\P$. In particular (by the convexity of the pressure map) $J$ contains the interval $(-M,-1)$. A similar argument allows us to construct potentials such that $J$ contains any compact subset of $(-\infty,0)$.

\begin{center}
\begin{tikzpicture}\label{fig3}

%linea intermitente
\draw[color=gray!50] [dashed](-2,1)--(2,1);
\draw[color=gray!50] [dashed](-6,1)--(-4,1);

%linea curva1
\draw[red] plot[smooth] coordinates {(-2,1) (0,2.5) (1,4)};

%linea curva2
\draw[red] plot[smooth] coordinates {(-6,3) (-5,1.5) (-4,1)};

%Ejes X
\draw [>= stealth, ->](-6,0)--(2,0) node[below=12pt,midway]{$\textmd{Figure 2: Phase transition for } F\in\F_s$};
\draw (2,0) node[below]{$t$};

%Ejes Y
\draw [>= stealth, ->](0,-0.2)--(0,4);
\draw (0,4) node[left]{$P(tF)$};

\draw (0,2.5) node{$\bullet$} node[right]{$\delta_{\Gamma}$};

\draw (0,1) node{$\bullet$} node[above right]{$\delta_\P$};

\draw (-3,1) node[above]{$J$};

%linea recta
\draw [domain=-4:-2]plot(\x,1);

\draw (-2,1) node{$\bullet$} node[below]{$t_F$};
\draw (-4,1) node{$\bullet$};

\end{tikzpicture}
\end{center}
 In light of this discussion we have the following definition.
\begin{definition}[Types of phase transition] A H\"older potential $F\in C_0(T^1M)$ exhibits a phase transition of \emph{type A} if the graph of the pressure map looks like Figure 1. A potential $F$ exhibits a phase transition of \emph{type B} if the graph of the pressure map looks like Figure 2.
\end{definition}

Type A and B phase transitions  represent basically all types of phase transitions for H\"older potentials in $C_0(T^1M)$.

\begin{remark}\label{rem5} The manifolds $(M_n)_n$ constructed in Theorem \ref{teo6} are all diffeomorphic to $M_1$. It is a well known fact that every parabolic subgroup of  $Iso(\H^N)$ is of divergence type: we can verify the hypothesis of Theorem \ref{teo6} if $(\widetilde{M},g)$ is isometric to $\H^N$. 
\end{remark}
A concrete situation where Theorem \ref{teo6} applies is given in the following corollary. 

\begin{corollary}\label{cor2} Let $M$ be a thrice-punctured sphere. We can endow $M$ with a complete hyperbolic metric for which it is possible to construct H\"older potentials exhibiting phase transitions (as described in Theorem \ref{teo6}). By Remark \ref{generalized} the same holds if $M$ is a $k$-punctured sphere and $k\ge 3$. 
\end{corollary}

\section{Phase transitions for the geometric potential}\label{cons}
In this section we will modify the metric at the cusp of a hyperbolic manifold in such a way that the geometric potential exhibits a phase transition. Since $F^{su}$ is not a potential that goes to zero through the cusps of the manifold, to apply the techniques developed in previous sections we need to consider a normalization. From now on we will assume that our manifold has real dimension $N$. 
\begin{definition}\label{norm} We define the \emph{normalized unstable jacobian} as the function $$U=F^{su}+(N-1).$$
\end{definition}
Our goal is to prove that under certain conditions $U$ exhibits a phase transition, just as the potentials in Theorem \ref{teo6}.
The construction starts with the hyperbolic space $\H^{N}$. It is convenient to think in the half space model, i.e. the space $\R^{N-1}\times \R^+$ with coordinates $(x_1,...,x_{N-1},x_0)$ and metric $$ds^2=\dfrac{1}{x_0^2}(dx_1^2+...+dx_{N-1}^2+dx_0^2).$$ To simplify notation we denote $(x_1,...,x_{N-1},x_0)=({\bf x},x_0)$. In this model we have a preferent point at infinity, we denote this by $\xi_\infty\in \partial_\infty\widetilde{M}$. We will modify the hyperbolic metric in a neighborhood of $\xi_\infty$. It will be  convenient to consider the diffeomorphism $\H^N\to \R^{N}$ taking $({\bf x},x_0)$ to $({\bf x},\log(x_0))$. In this model the hyperbolic metric takes the form $g =e^{-2t}d{\bf x}^2+dt^2$. For a positive function $T:\R\to\R$, we define the Riemannian metric $$g_T=T(t)^2d{\bf x}^2+dt^2.$$ The sectional curvature of $g_T$ has value $-(T''(t)/T(t))$ for the planes generated by $\langle\partial/\partial x_i,\partial/\partial x_j\rangle$ and value $-(T''(t)/T(t))^2$ for those generated by $\langle\partial/\partial x_i, \partial/\partial t\rangle$. Define 
$$K(t)=-\frac{T''(t)}{T(t)}.$$
Bounds on $K(t)$ clearly imply bounds on the curvature of $g_T$. The lines $t\mapsto ({\bf x},t)$ are still geodesics and any isometry of $\R^{N-1}$ acts isometrically on $(\H^{N},g_T)$, where the action is given by $A.({\bf x},t)=(A({\bf x}),t)$. Observe that translations act transitively in  $H_t=\{({\bf x},t): {\bf x}\in \R^{N-1}\}$. This two basic observations and the definition of the Busemann function are enough to conclude that $H_t$ are the horospheres associated to $\xi_\infty$. For a function constant on horospheres we  will use the notation $F(t)=F({\bf x},t)$. In this section the \emph{height of a segment} will be the maximum value of the $t$ coordinate over the segment.
%It will be convenient to shift $F$, from now on when we write $F(t)$ we actually mean $F(t-B_1)$, with this convention $F(0)=1$ and $F$ decrease for $t\ge 0$. 
%Observe that if $K_{\max}(t)\le -\left(1+\frac{F(t)}{(n-1)}\right)^2$, then $F^{su}\le -(n-1)(1+\epsilon)$ or equivalently $F^*\le -(n-1)\epsilon$. We want to make the pinching as we approach $\xi_\infty$ to be compatible with $F$.
%Since $T$ is positive we define $L$ by the equation $T(t)=e^{-L(t)}$. Then $T''(t)/T(t)=(L'(t))^2-L''(t)$. We will construct $L'$ satisfying that $T''(t)/T(t)\ge F(t)$\\ 
Suppose we have  a surjective, strictly increasing function $u:(0,\infty)\to\R$. We can define $T=T(u)$ by the equation $$T(u(t))=1/t.$$ In this context we will use  $o=({\bf 0},u(1))$ as reference point. Let $p$ be a translation in $\R^N$ such that $d(o,po)=1$. It is proven in \cite[Section 3]{dop} that there exists a uniform constant $C$ such that $$|d_T(o,p^n o)-2u(|n|)|\le C.$$ It will be important for us the fact that $C$ only depends on the pinching of the  sectional curvature of $(\H^N,g_T)$. Using the symmetry of our metric we can also conclude that $$|t_n-u(|n|))|\le D,$$ where $t_n$ is the maximum height of the geodesic segment $[o,p^n o]$. Similarly to what happened with $C$, the constant $D$ only depends on the pinching of the metric. From now on $D$ and $C$ will be the constants associated to a metric with pinching $-(1/3)^2\ge K_g\ge -2^2$.
Observe that 
\begin{align*}
K(u(t))&=-\dfrac{2tu'(t)+t^2u''(t)}{(u't)^3},\\
&=-\dfrac{1+2t\phi'(t)+t^2\phi''(t)}{(1+t\phi'(t))^3},\\
&=-\dfrac{g_1(t)+tg_1'(t)}{g_1(t)^3},\\
&=-(g_2(t)-\frac{t}{2}g_2'(t)),
\end{align*}
where we have made the substitutions \begin{align}\label{kk}u(t)=\log(t)+\phi(t); g_1(t)=1+t\phi'(t) \text{ and }g_2(t)=\frac{1}{g_1(t)^2}.\end{align}

\subsection{Construction of a special metric at the cusp}\label{cons1}
We will start by constructing a function $g_2$ satisfying several properties. Using (\ref{kk}), this will give us a function $u$ and therefore a metric on $\R^N$.

 Define $a_{n+1}=(1+\frac{1}{(N-1)n})^2$, $b_n=1-(2a_n-1)^{-\frac{1}{2}}$ and $c_n=(a_{n-1}-a_n)$. For $n\ge 2$ we choose $k(n)\in\N$ such that $$\exp \bigg(\left(\frac{1}{n}+\frac{1}{2}\right)D\bigg)\sum_{|k|\ge k(n)}\exp \bigg(-\left(\frac{1}{2}+\frac{1}{n}\right)2\log(|k|)\bigg)\le \frac{1}{n^2}.$$
Without loss of generality we can assume that $(k(n))_{n}$ is strictly increasing. We will define a sequence $(p_n)_n$ so that the conditions below are satisfied. We do this by induction, i.e.  the choice of $p_1,...,p_n$  will determine $p_{n+1}$. For $n\ge 2$ define $\Delta_{n}=p_n-p_{n-1}$.

\begin{enumerate}
\item\label{1o} The sequence $\left(\dfrac{c_n}{\Delta_{n}}\right)$ is a strictly decreasing.
\item \label{2o}$b_n\log(p_{n+1})\ge -b_n\log (p_n)+\sum_{i=2}^{n-1} b_i\log (p_{i+1}/p_i)$, for $n\ge 3$.
\item \label{3o}$(1-2b_n)\log(p_{n+1})\ge \log(k(n+1))$, whenever $1-2b_n>0$.
\item \label{4o}$\lim_{n\to 0}\dfrac{p_n c_n}{\Delta_{n}}=0$. 
\end{enumerate}

For $n\ge 2$ define the line connecting the points $(p_n,a_n)$ and $(p_{n+1},a_{n+1})$ as $J_n$. We could have assumed that $p_2$ is big enough compared to $p_1$ so that $J_2(0)\le 2$, we will do assume that. Define the intervals $I_n=[p_{n-1},p_n]$.
 We will construct a $\mathcal{C}^\infty$ function $g_2:\R^+\to\R$ satisfying the following properties:
\begin{enumerate}
\item\label{11}  For $t\in I_n$ and $n\ge 3$ we have that $2a_{n-1}-1\ge g_2(t)\ge a_n$.
\item\label{22} $g_2''(t)\ge 0\ge g_2'(t)$, for $t\ge 2$.
\item\label{33} $g_2(t)-\frac{t}{2}g_2'(t)\in [1/3,2]$, for all $t\in \R^+$.
\item\label{44} For $t\in I_n$ and $n\ge 3$ we have $g_2(t)-\frac{t}{2}g_2'(t)\ge a_n$.
\item\label{55} $g_2(t)=1$, for $t<1$.
\end{enumerate}
Observe that $(g_2-(t/2)g_2')'=g_2'/2-(t/2)g_2''$, so the condition $g_2''\ge 0\ge g_2'$ implies that  $g_2(t)-(t/2)g_2'(t)$ is non-increasing. Notice that if $0\ge g_2'$, then $g_2(t)-\frac{1}{2}tg_2'(t)\ge g_2(t)$. In particular, if $g_2\ge a_n$, then the same holds for $g_2(t)-\frac{1}{2}tg_2'(t)$. We now explain how to construct $g_2$. First define $J_0:\R^+\to\R$ as $J_0(t)=\sup_{n\ge 2}J_n(t)$. Now define
$$J(t) = \left\{
	\begin{array}{ll}
		1, \text{ if } t\in (0,1]\\
		\min \{J_0(t),t\}, \text{ if }t\ge 1
	\end{array}\right.$$
The function $J$ is not smooth, but we can smooth out $J$ in a neighborhood of the nodes to obtain a smooth function $J^*$ (as close to $J$ in the  $\mathcal{C}^\infty$ topology as needed). When $t\ge 2$ we can assume that $J^*\ge J$ and that $J^*$ is convex decreasing. The fact that $J^*$ can be taken convex on that region comes from the condition (\ref{1o}) in the definition of $(p_n)_n$. We remark that the choice of $(2a_{n-1}-1)$ as the upper bound of $g_2$ is done just to get room for this perturbation (notice $2a_{n-1}-1>a_{n-1}$). We finally set $g_2=J^*$. Define a function $\phi$ (up to additive constant) by the equation $$\phi'(t)=(g_2^{-1/2}-1)/t.$$ 
First observe that  $g_2\ge 1$, implies that $0\ge \phi'$. For $n\ge 3$ and $t\in I_n$ we have $2a_{n-1}-1\ge g_2(t)$, therefore $\phi'(t)\ge -b_{n-1}/t$. Since $(b_n)_n$ is a positive decreasing sequence we actually have that $\phi'(t)\ge -b_{n-1}/t$, for every $t\ge p_{n-1}$. Then for $n\ge 3$, and $t\ge p_{n+1}$ we get
\begin{align*}
\phi(t)-\phi(p_2)&=\sum_{i=2}^{n-1} \int_{p_i}^{p_{i+1}} \phi'(s)ds+\int_{p_{n}}^t \phi'(s)ds \\
& \ge \sum_{i=2}^{n-1} \int_{p_i}^{p_{i+1}} -\dfrac{b_i}{s} ds+\int_{p_{n}}^t -\dfrac{b_n}{t}ds\\
& = -\sum_{i=2}^{n-1} b_i\log(p_{i+1}/p_i)-b_n\log(t)+b_n\log(p_n)\\
& \ge -b_n\log(p_{n+1})-b_n\log(t)\\
& \ge -2b_n\log(t).\\
\end{align*}

We will normalize $\phi $ so that $\phi(p_2)=0$. Observe that by making the $p_i$'s even bigger we obtain that $\lim_{t\to\infty}\phi(t)=-\infty$. We will make that assumption. Finally define $u:\R^+\to \R$, by the equation $$u(t)=\log(t)+\phi(t).$$ 
By definition of $\phi'$ we have that $ u'=\frac{1}{t}+\phi'(t)=\frac{g_2^{-1/2}}{t}$, therefore $u$ is surjective and strictly increasing (recall $g_2\le 2$). As commented at the beginning of this section,  a function with the properties of $u$ determine a function $T=T(u)$ and therefore a metric $g_T$. We now pick the reference point  $o=({\bf 0},u(1))$ and a parabolic isometry such that $d(o,p o)=1$. We will now state a number of observations which will lead to the proof of Theorem \ref{i2}. \\
\noindent\\
{\bf Observation 1:}  The formula $K(u(t))=-(g_2(t)-(t/2)g_2'(t))$, and property (\ref{33}) in the definition of $g_2$ implies that $K\in [-2,-\frac{1}{3}]$ In particular the curvature of the metric $g_T$ satisfies the pinching $$ K_{g_T}\in [-4,-\frac{1}{9}].$$ 

\noindent
{\bf Observation 2:}  Hypothesis (\ref{44}) in the definition of the sequence $(p_n)_n$, and the calculation of $J^*(t)-\frac{1}{2}t(J ^*)'(t)$ on the intervals $I_n$ gives us that $$\lim_{t\to\infty} K(t)=-1.$$ 
For big enough $t$, $K(t)$ increases to $-1$. Remark \ref{insp} implies that the function $U$ goes to zero as $t$ goes to infinity. Moreover, it implies that $U$ is negative in a neighborhood of $\xi_\infty$. \\\\
\noindent
{\bf Observation 3:}   By property (\ref{44}) in the definition of $g_2$ we know that for $t\le p_{n+1},$ we have $$K(u(t))\le -a_{n+1}.$$ Combining this and  Remark \ref{insp} we get that for every $t<u(p_{n+1}),$ we have that  $ U(t)\le -\frac{1}{n}$. \\\\
\noindent
{\bf Observation 4:} Using hypothesis (\ref{3o}) in the definition of the sequence $(p_n)_n$ and the lower bound for $\phi(t)$ we get 
 \begin{align*}\log(k(n+1))<&(1-2b_n)\log(p_{n+1})\\<&\log(p_{n+1})+\phi(p_{n+1})=u(p_{n+1}).\end{align*}

\begin{lemma}\label{lemma3} There exists $m\in\N$ such that for every $n\ge k(m)$ and $|k|\le k(n+1)$, the function $U$ is at most $-\frac{1}{n}$  on the geodesic segment $[o,p^k o]$.
\end{lemma}
\begin{proof} It follows from {\bf Obs. 1} that if $|k|\le k(n+1)$, then the  height of $[o,p^k o]$ is at most $u(k(n+1))+D$. For $n$ big enough this is less than $\log(k(n+1))$. Using {\bf Obs. 4} we conclude that if $|k|\le k(n+1)$, then the  height of $[o,p^k o]$ is at most $u(p_{n+1})$. Finally {\bf Obs. 3} gives us that $U\le -\frac{1}{n}$ on the geodesic segment $[o,p^k o]$ if $|k|\le k(n+1)$.
\end{proof}

\begin{lemma}\label{lemma4} The critical exponent of $\P=\langle p\rangle$ is equal to $1/2$. Moreover, $\P$ is of divergence type.\end{lemma}
\begin{proof}As explained in \cite[Section 3]{dop}, the Poincar\'e series of $\P$ for the metric $g_T$ is equivalent to the series $$\sum_{n\in\Z}\exp (-2su(|n|)).$$
Observe that for every $\epsilon>0$ there exists a natural number $N$ such that if $n\ge N$, then we have $\phi(n)>-\epsilon \log(t)$. In particular  
$$\sum_{n\ge N}\exp (-2su(|n|))\le \sum_{n\ge N}\exp (-2s(1-\epsilon)\log(|n|) ).$$
If $s(1-\epsilon)>1/2$, then the right hand side converges. This implies that for every $\epsilon>0$ the following inequality holds: $\delta_\P\le \frac{1}{2(1-\epsilon)}$. On the other hand for big $n$ we have $u(n)<\log(n)$. Then 
$$\sum_{n>N'} \exp (-2s\log(|n|) )<\sum_{n>N'}\exp (-2su(|n|)).$$
Since the critical exponent of the left hand side is $1/2$ (and of divergence type) we get the inequality $\delta_\P\ge 1/2$. We conclude that $\delta_\P=1/2,$  and that $\P$ is of divergence type.

\end{proof}

\begin{lemma}\label{convU} The pair $(\P,U)$ is of convergence type with respect to the metric $g_T$.
 \end{lemma}
\begin{proof}
We denote by $d_T$ the distance function induced by $g_T$. Combining Lemma \ref{lemma3} and the definition of $k(n)$ we get the inequality
\begin{align*}
&\sum_{|k|\ge k(m)}\exp\bigg(\int_o^{p^k o}U-\frac{1}{2} d_T(o,p^k o)\bigg)\\
=&\sum_{n\ge m} \sum^{k(n+1) -1}_{k=k(n)}\exp\bigg(\int_o^{p^k o}U-\frac{1}{2} d_T(o,p^k o)\bigg)\\
\le &\sum_{n\ge m}\sum^{k(n+1) -1}_{k=k(n)}\exp\bigg(-\left(\frac{1}{n}+\frac{1}{2}\right) d_T(o,p^k o)\bigg)\\
\le &\sum_{n\ge m} \sum^{k(n+1) -1}_{k=k(n)}\exp\bigg(-\left(\frac{1}{n}+\frac{1}{2}\right) (2u(|k|))\bigg)\exp\bigg(\left(\frac{1}{n}+\frac{1}{2}\right)D\bigg)\\
\le &\sum_{k=1}\frac{1}{k^2}.
\end{align*}
{\bf Obs. 2} and Lemma \ref{lem} implies that $\delta_\P=\delta_\P^U$ and then by Lemma \ref{lemma4} we know that the series above is exactly the Poincar\'e  series associated to $(\P,U)$. This finishes the proof of the lemma.
\end{proof}

\subsection{Construction of the family $\{M_{n,m}\}_{n,m}$}\label{sub2}
We have now  all the ingredients to construct the Riemannian manifold announced in Theorem \ref{i2}. We will use the notation introduced at the beginning of this section. We start with $(\R^N,g_{hyp})$, where $g_{hyp}$ is the hyperbolic metric, and the function $u$ constructed in Section \ref{cons1}. We choose a hyperbolic isometry $h$ (for the hyperbolic metric) so that $H=\langle h\rangle$ is in Schottky position with respect to $\P=\langle p\rangle$. We moreover assume that the axis of $h$ has height smaller that $u(1/2)$. Define $\Gamma$ as the group generated by $p$ and $h$ and let $M=\R^N/\Gamma$. The closed geodesic associated to $h$ is denoted by $\gamma=\gamma_h$. We can `cut' the cusp associated to $\P$ above height $u(1/2)$ and replace it by the cusp endowed with the metric $g_T$. This is possible because $g_T$ is the hyperbolic metric on the region $\{({\bf x},t):t< u(1)\}$. We have constructed a new Riemannian metric $g$ on $M$. We lift the metric to the universal cover; this is our new Hadamard manifold $(\widetilde{M},g)$. We will check that the Riemannian manifold $(M,g)$ satisfies the properties announced in Theorem \ref{i2}. Since the geometric structure has change we will be careful with our notation. The generator of the  parabolic subgroup of $Iso(\widetilde{M},g)$ corresponding to the cusp is denoted by $p_*$ and $h_*$ is the hyperbolic isometry associated to the closed geodesic $\gamma$ in $M$. The group generated by $p_*$ is denoted by $\P_*$ and the group generated by $h_*$ is denoted by $H_*$. The geometric potential of $M$ is denoted by $F^{su}$, and its normalization by $U$ (see Definition \ref{norm}). As before, we will organize the relevant information in a couple of observations and lemmas. It will be convenient to define $Q=-U$. We start with the following definition.

\begin{definition}\label{man} We denote by $\Gamma_{n,m}$  the group generated by $h_*^n$ and $p_*^m$. Let $M_{n,m}$ be the covering of $M$ associated to the subgroup $\Gamma_{n,m}$ of $\Gamma$. The lift of a potential $G$ on $T^1M$ to $T^1M_{n,m}$ is denoted by $G_{n,m}$.
\end{definition}
\noindent\\
{\bf Observation 5:} The closed geodesic $\gamma$ lies in the region where $g$ is hyperbolic. This implies that $U$ and $Q$ vanish along $\gamma$. 

\noindent\\
{\bf Observation 6:} Since $F^{su}$ is locally defined in terms of $g$ and every local structure is preserved under taking coverings, we conclude that $F^{su}_{n,m}$ is the geometric potential of $M_{n,m}$. 

\noindent\\
{\bf Observation 7:} The potential $Q$ vanishes at infinity and it is positive in a neighborhood of the cusp associated to $\P_*$. This follows directly from {\bf Obs. 2}. 

\begin{lemma}\label{convF}The pair $(\P_*,-tQ)$  is of convergence type for every $t\ge 1$.
\end{lemma}

\begin{proof} The reference point used in the proof of Lemma \ref{convU} lies in the piece of $M$ coming from $(\R^N,g_T)$. Since horoballs are convex it follows that $d_T=d$ on that region. This implies that the series estimated in Lemma \ref{convU} is exactly the Poincar\'e series associated to $\P_*$, which implies that $(\P_*,-Q)$ is of convergence type. By the construction of $g_2$ we know that $Q$ is positive above height $u(2)$ (see {\bf Obs. 7}). Change the reference point to $o'$, a point with height $u(2)$. By convexity of this region and the definition of $g_2$ we get that $-t\int_{o'}^{p^ko'}Q\le -\int_{o'}^{p^ko'}Q$, for every $t\ge 1$. Plugging this into the Poincar\'e series of $-tF$ and $-F$ implies the lemma.
\end{proof}

The idea now is to apply Theorem \ref{teo6} to the potential $Q=-U$. Lemma \ref{convF} and {\bf Obs. 7} implies that $Q$ belongs to the class $\F_s$. Despite that $Q$ is not non-negative, by {\bf Obs. 5} and  the discussion at the end of Section \ref{pha},  we know that $F_n$ exhibits a phase transition for sufficiently big $n$. This is not a very satisfactory answer, one would like to know which type of phase transition we encountered, either type A or type B phase transition. By the construction of the modified metric at the cusp, a phase transition of type B seems unlikely to occur. Since we are not able to completely rule out that case, we present an argument to justify that a type A phase transition is always possible to construct. As before, we denote by $\widetilde{Q}$ to the lift of $Q$ to the universal cover of $M$. Denote by $\xi\in \partial_\infty \widetilde{M}$ to the parabolic fixed point of $\P_*$. By construction of $g_2$ we know that there exist real numbers $s$ and $L$ such that $\widetilde{Q}$ is negative precisely at vectors whose  base lies in the interior of $B_\xi(s)\setminus B_\xi(s+L)$ and it is zero for vectors with base in $M\setminus B_\xi(s)$. It follows easily from this fact that there exists $m_0$ such that if $m\ge m_0$, then 
\begin{align}\label{y} \int_o^{g.o} \widetilde{Q}>0,\end{align}
 for every $g\in \Gamma_{1,m}$, in particular for every $g\in \Gamma_{n,m}$. It follows from the definition of the Poincar\'e series and the fact that $Q$ goes to zero through the cusp of $M$ (see {\bf Obs. 2}), that if the pair  $(\Gamma_{n,1},-Q_{n,1})$ is of convergence type with $\delta_{\Gamma_{n,1}}^{-Q_{n,1}}=\delta_\P$, then the same holds for $(\Gamma_{n,m},-Q_{n,m})$. Moreover, if $m\ge m_0$, then inequality (\ref{y}) and the definition of the Poincar\'e series implies that if $(\Gamma_{n,m},-Q_{n,m})$  is of convergence type with $\delta_{\Gamma_{n,m}}^{-Q_{n,m}}=\delta_\P$, then the same holds for  $(\Gamma_{n,m},-tQ_{n,m})$, for every $t\ge 1$. Combining {\bf Obs. 6} and the discussion above we conclude that $Q_{n,m}$ admits a phase transition of type A for big enough $n$ and $m$.

\subsection{Conclusions}
In  Section \ref{sub2} we constructed a family of geometrically finite negatively curved manifolds $M_{n,m}$ (see Definition \ref{man}) for which $Q_{n,m}$ exhibits a phase transition of type A if $n$ and $m$ are big enough. Since 
$$tQ_{n,m}=-tF_{n,m}^{su}-t(N-1),$$
 the existence of an equilibrium state for $tQ_{n,m}$ is equivalent to the existence of one for $-tF_{n,m}^{su}$. 
Just as in Theorem \ref{teo6}, there exists $t_{n,m}\in [-1,0)$ such that $tQ_{n,m}$ has an equilibrium state for all $t>t_{n,m}$, and there is not equilibrium state for $t<t_{n,m}$. We remark that $F_{n,m}^{su}$ does not admit an equilibrium state since $(\Gamma_{n,m},-Q_{n,m})=(\Gamma, F_{n,m}^{su}+(N-1))$ is of convergence type.  It is clear from the construction that $M_{n,m}$ are extended Schottky manifolds. By Lemma \ref{lemma4} we know that $\P_*$ and $\langle p_*^n\rangle$ are  of divergence type, then using Theorem \ref{gapdiv} we obtain $\delta_{\langle p_*^n\rangle}<\delta_{\Gamma_{n,m}}$. We conclude that  the geodesic flow on $M_{n,m}$ has a measure of maximal entropy. All this together gives us the following result.

\begin{theorem}\label{fin} Each $M_{n,m}$ is an extended Schottky manifold and $F^{su}_{n,m}$ is the geometric potential of $M_{n,m}$. Suppose that $n$ and $m$ are sufficiently large. Then there exists $s_{n,m}\in (0,1]$ such that the following holds:
\begin{enumerate}
\item $tF^{su}_{n,m}$ has an equilibrium state for $t<s_{n,m}$. 
\item $tF^{su}_{n,m}$ does not has an equilibrium state for $t>s_{n,m}$.
\item The pressure map is linear in $(s_{n,m},\infty)$.  
\end{enumerate}
\end{theorem}

\begin{center}
\begin{tikzpicture}\label{fig4}

%linea intermitente

\draw[color=gray!50] [dashed]plot[smooth] coordinates {(-2,0.5) (-3.5,0.5)}; 

%linea curva
\draw[red] plot[smooth] coordinates {(-2,0.5) (-2.8,1.2)(-3.7,2.7) (-4,4)};

%Ejes X
\draw [>= stealth, ->](-6,1)--(1,1) ;
\draw (1,1) node[below]{$t$};

%Ejes Y
\draw [>= stealth, ->](-3.5,-0.2)--(-3.5,4);
\draw (-3.5,4) node[right]{$P(tF_{n,m}^{su})$};

\draw (-3.5,2.3) node{$\bullet$} node[right]{$\delta_{\Gamma_{n,m}}$};

%\draw (0,1.48) node{$\bullet$} node[below right]{$\delta_\P$};

\draw[black] plot[smooth] coordinates {(0,-0.1)(-2,0.5)};

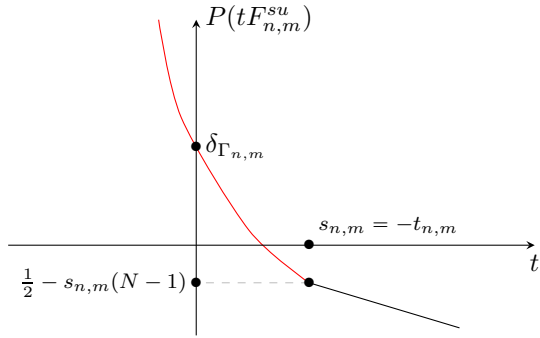
\node  at (-2.1,-1) {Figure 3: Phase transition for $P(tF_{n,m}^{su})$};

\draw (-2,0.5) node{$\bullet$};
\draw (-2,1) node{$\bullet$} node[above right]{\footnotesize $s_{n,m}=-t_{n,m}$};
\draw (-3.5,0.5) node{$\bullet$} node[left]{\footnotesize $\frac{1}{2}-s_{n,m}(N-1)$};
\end{tikzpicture}
\end{center}

As mentioned in Remark  \ref{generalized} we could have done the same construction for an extended Schottky manifold with one parabolic and arbitrary number of hyperbolic generators (modulo taking big powers of the parabolic and hyperbolic generators). Combining Remark \ref{rem5}  and the proof of Theorem \ref{fin} we obtain the following result.

\begin{corollary}\label{cor3} Let $M$ be a $k$-punctured sphere with $k\ge 3$. We can endow $M$ with a complete Riemannian metric with pinched negative sectional curvature, such that the geometric potential exhibits a phase transition of type A. Moreover, the geodesic flow of $M$ has an unique measure of maximal entropy, and the Riemannian metric is hyperbolic outside a neighborhood of one of the punctures.
\end{corollary}

%\section{Conclusions}

%We therefore found a manifold with geometric potential exhibiting a phase transition of type A.  Recall the following theorem
%\begin{theorem} \cite{pps} \label{PPS2}
%Let $\widetilde{X}$ as in the hypothesis of Theorem \ref{PPS1}. If the geodesic flow of M is conservative with respect to the Liouville measure and $\delta_\Gamma^{F^{su}}\le 0$, then the Liouville measure is proportional to the Gibbs measure of the potential $F^{su}$. Furthermore, $(\Gamma,F^{su})$ is of divergence type and $\delta_\Gamma^{F^{su}}=0$.
%\end{theorem}
%\noindent
%We remark that the hypothesis: $\delta_\Gamma^{F^{su}}\le 0$ can be removed. This was recently proved by F. Riquelme in \cite{r}. Sinnce we do not have equilibrium measure for $F^{su}$, the geodesic flow can not be conservative. This is not  surprising, the convex core of our manifold is very narrow compared with $X$. 

\end{document}